\documentclass[a4paper, 10pt]{amsart}
\pdfoutput=1

\def\g{\ensuremath{\mathfrak{g}}}
\def\p{\ensuremath{\mathfrak{p}}}
\def\q{\ensuremath{\mathfrak{q}}}

\usepackage{tikz,tikz-cd}

\usepackage{amsmath}
\usepackage{amssymb}
\usepackage{amsthm}
\usepackage{graphicx}
\usepackage[top=1.2in, left=1.2in, right=1.2in, bottom=1.2in]{geometry}
\usepackage[all,cmtip]{xy}
\newtheorem{theorem}{Theorem}[section]
\newtheorem{proposition}{Proposition}[section]
\newtheorem{lemma}{Lemma}[section]
\newtheorem{corollary}{Corollary}[section]
\newtheorem{definition}{Definition}[section]

\theoremstyle{remark}
\newtheorem{remark}{Remark}[section]

\newcommand{\der}{{\rm d}}
\newcommand{\dz}{\wedge}
\newcommand{\be}{\begin{equation}}
\newcommand{\ee}{\end{equation}}
\newcommand{\bma}{\begin{pmatrix}}
\newcommand{\ema}{\end{pmatrix}}
\newcommand{\soa}{\frak{so}}
\newcommand{\Om}{\Omega}
\newcommand{\om}{\omega}
\newcommand{\Span}{\mathrm{Span}}
\newcommand{\bbR}{\mathbb{R}}

\newcounter{mnotecount}[section]
\renewcommand{\themnotecount}{\thesection.\arabic{mnotecount}}

\newcommand{\mnote}[1]%{}%
{\protect{\stepcounter{mnotecount}}$^{\mbox{\footnotesize
$%\!\!\!\!\!\!\,
\bullet$\themnotecount}}$ \marginpar{\color{red}
\raggedright\tiny\em
$\!\!\!\!\!\!\,\bullet$\themnotecount: #1} }

\newcommand{\hook}{\raisebox{-0.35ex}{\makebox[0.6em][r]
{\scriptsize $-$}}\hspace{-0.15em}\raisebox{0.25ex}{\makebox[0.4em][l]{\tiny
 $|$}}}

\title{New relations between $\mathrm{G}_2$-geometries in dimensions $5$ and $7$}
 \author{Thomas Leistner} \address{School of Mathematical Sciences, University of Adelaide, SA 5005, Australia} \email{\tt thomas.leistner@adelaide.edu.au}
\author{Pawe\l~ Nurowski} \address{Centrum Fizyki Teoretycznej,
Polska Akademia Nauk, Al. Lotnik\'ow 32/46, 02-668 Warszawa, Poland}
\email{nurowski@cft.edu.pl}
\author{Katja Sagerschnig} \address{Faculty of Mathematics, University of Vienna, Oskar-Morgenstern-Platz 1, A-1090 Wien, Austria}
\email{katja.sagerschnig@univie.ac.at}
\thanks{This work was supported  by the  Australian Research
Council via  the grants FT110100429 and DP120104582, by the Polish National Science Center (NCN) via DEC-2013/09/B/ST1/01799, and by  the Austrian Science Fund (FWF) via grant J 3071-N13.}

\date{\today}

\begin{document}

\begin{abstract}
There are two well-known parabolic split $G_2$-geometries in dimension five, $(2,3,5)$-distributions and $\mathrm{G}_2$-contact structures. Here we link these two geometries with yet another $G_2$-related  contact structure, which lives on a seven-manifold.
We present a natural geometric construction of a Lie contact structure on a seven-dimensional  bundle over a five-manifold endowed with a $(2,3,5)$-distribution. For a class of distributions  the induced Lie contact structure is constructed explicitly and we determine its symmetries. We further study the relation between the canonical normal Cartan connections associated with  the two structures. In particular, we show that the Cartan holonomy of the induced Lie contact structure reduces to $\mathrm{G}_2$. 
Moreover, the curved orbit decomposition associated with a $\mathrm{G}_2$-reduced Lie contact structure on a seven-manifold is discussed. It is shown that in a neighbourhood of each point on the open curved orbit the structure descends to a $(2,3,5)$-distribution on a local leaf space, provided an additional curvature condition is satisfied. The closed orbit carries an induced $G_2$-contact structure.
%$(2,3,5)$-distributions and $\mathrm{G}_2$-contact structures, respectively.
\end{abstract}

\maketitle
\section{Introduction}
A distribution $\mathcal{D}$
%=\mathrm{span}(\xi_1,\,\xi_2 )$ 
(locally) spanned by vector fields $\xi_1$ and $\xi_2$ on a five manifold $M$
is called a $(2,3,5)$-distribution, or generic,  if the five vector fields
\begin{align*}
\begin{aligned}
\xi_1,\,\xi_2,\, [\xi_1,\xi_2],\, [\xi_1,[\xi_1,\xi_2]],\, [\xi_2,[\xi_1,\xi_2]]
\end{aligned}
\end{align*}
are linearly independent at each point. Two distributions $\mathcal{D}$ and $\mathcal{D}'$ are said to be   equivalent, if there is a  diffeomorphism $\phi:M\to M'$ such that $\phi_*\, \mathcal{D}=\mathcal{D}'$.
It is a classical result \cite{goursat}, that  every $(2,3,5)$ distribution is locally equivalent to
the kernel $\mathcal{D}_F=\mathrm{ker}(\omega_1,\omega_2,\omega_3)$ of three one-forms
\begin{align*}
\begin{aligned}
\omega_1= \der y - p\, \der x, \ \omega_2=\der p - q\, \der x, \ \omega_3 = \der z - F \der x
\end{aligned}
\end{align*}
for coordinates $(x,y,p,q,z)$ on an open subset $U\subset \mathbb{R}^5$ around the origin and a smooth function $F=F(x,y,p,q,z)$ such that $F_{qq}\neq 0$.
%The requirement that $\mathcal{D}$ is $(2,3,5)$ precisely means that the function $F$ satisfies the genericity condition 
Infinitesimal symmetries of a distribution $\mathcal{D}\subset T M$  are vector fields $\eta\in\mathfrak{X}(M)$ that preserve the distribution, i.e. $\mathcal{L}_{\eta}\xi= [\eta, \xi] \in\Gamma(\mathcal{D})$ for all $\xi\in\Gamma(\mathcal{D})$. Cartan and Engel found, independently but at the same time \cite{cartan1, engel1}, the first explicit realization of the  exceptional simple Lie algebra $\g_2$. They realized it as the Lie algebra of infinitesimal symmetries of a  rank $2$ distribution that is locally equivalent
to the  distribution $\mathcal{D}_{q^2}$ associated with the function $F=q^2$.
Cartan's later fundamental work \cite{cartan-cinq} shows that the Cartan-Engel distribution $\mathcal{D}_{q^2}$ can indeed be regarded as the \emph{flat} and \emph{maximally symmetric} model of $(2,3,5)$ distributions. Flat, because he shows how to associate to any $(2,3,5)$ distribution a curvature tensor, called Cartan quartic $\mathcal{C}\in\Gamma(S^4 \mathcal{D}^*)$, which vanishes if and only if the distribution is locally equivalent to $\mathcal{D}_{q^2}$. Maximally symmetric, because  he proves that the symmetry algebra of a distribution with nonvanishing Cartan quartic has dimension smaller than $\mathrm{dim}(\mathfrak{g}_2)=14$.

More recent work \cite{nurowski-metric} associates to a $(2,3,5)$ distribution a canonical conformal structure of signature $(2,3)$, i.e., an equivalence class  of pseudo-Riemannian metrics of signature $(2,3)$ where two metrics  $g$ and $\hat{g}$ are considered equivalent if one is a conformal rescaling of the other, meaning that $\hat{g} = e^{2f} \, g$. On the one hand, this allows to understand the geometry of $(2,3,5)$ distributions in terms of the more familiar conformal geometry. On the other hand, the construction provides an interesting class of conformal metrics given explicitly in terms of a single function $F$, see \cite{nurowski-hol, leistner-nurowski, graham-willse, travis}.  
From an algebraic point of view, the construction is based on the Lie algebra inclusion $\mathfrak{g}_2\hookrightarrow \mathfrak{so}(3,4)$,  see \cite{mrh-sag-rank2}. 

The present article is of a similar flavour. It is motivated by the  observation that, while $(2,3,5)$ distributions and conformal structures are among the most prominent geometries associated with the Lie algebras $\mathfrak{g}_2$ and $\mathfrak{so}(3,4)$, respectively, they are not the only  ones.  
In particular, associated with every simple Lie algebra, there is a \emph{parabolic contact geometry}. A  parabolic contact geometry is given by a contact distribution (i.e., a corank one distribution that is locally  given as the kernel of a one-form  $\theta$  such that  $\theta\wedge (d\theta)^n\neq 0$) and additional geometric structure on the contact distribution. By Pfaff's theorem, all contact distributions are locally equivalent. However, by equipping the distribution  with additional geometric structure, e.g., with a tensor field of some type or a decomposition of the contact distribution as a tensor product of  auxiliary vector bundles, one again obtains an interesting geometry with non-trivial local invariants.  Every  parabolic contact geometry has a flat and maximally symmetric model and the infinitesimal symmetry algebra of the model realizes  the simple Lie algebra in question.  
Parabolic contact geometries associated with special orthogonal Lie algebras $\mathfrak{so}(p+2,q+2)$ have been studied under the name Lie contact structures, \cite{Sato-yama1,Sato-yama2, miyaoka1, miyaoka2}, by means of Tanaka theory \cite{tanaka}. They are known to appear naturally on unit sphere tangent bundles over Riemannian manifolds; the model is related to the classical Lie sphere geometry of oriented hypersheres in a sphere.

Now a natural question arises:  Can we use the Lie algebra inclusion $\mathfrak{g}_2\hookrightarrow \mathfrak{so}(3,4)$ to relate $(2,3,5)$-distributions to the $\mathfrak{so}(3,4)$-Lie contact geometry, as we did for the construction of conformal structures from $(2,3,5)$ distributions? Inspecting the models of the two geometries shows that there is indeed such a natural geometric construction. More precisely, in Section \ref{secContruction} we  show the following:

\begin{theorem}\label{thm1.1}
Let $\mathcal{D}=\mathrm{span}(\xi_1,\,\xi_2 )$ be a $(2,3,5)$ distribution with derived rank $3$ distribution $[\mathcal{D},\mathcal{D}]=\mathrm{span}(\xi_1,\,\xi_2,\, [\xi_1,\xi_2]\, )$ and consider the $7$ manifold $\mathbb{T}= \mathbb{P}([\mathcal{D},\mathcal{D}]) \setminus  \mathbb{P}(\mathcal{D})$ of lines in the rank $3$ distribution transversal to the rank $2$ distribution. Then $\mathbb{T}$ carries a naturally induced Lie contact structure. The induced Lie contact structure is flat if and only if the $(2,3,5)$ distribution is flat.
\end{theorem}
The proof of the theorem is based on the equivalent descriptions of $(2,3,5)$-distributions and Lie contact structures, respectively, as particular types  of Cartan geometries. It employs a functorial construction that assigns to the canonical Cartan geometry encoding a $(2,3,5)$-distribution a Cartan geometry encoding a Lie contact structure.

 In Section \ref{secExamples} we use the structure equations for  a class of $(2,3,5)$ distributions (for those that are encoded in terms of functions $F=h(q)$ of a single variable $q$) to construct the corresponding Lie contact structure  explicitly in terms of a conformal symmetric rank $4$ tensor on the contact distribution. In particular, this enables us to find explicit generators for the symmetry algebras in the case that $F=\tfrac{1}{k(k-1)}q^k$.
%of $F$ and its derivatives.

In Section \ref{secReduction} we analyze  the relation between the canonical normal Cartan connections associated with the two structures. We show that the construction preserves normality, see Lemma \ref{normal}, and as a consequence:
\begin{proposition}\label{prop-hol}
 The holonomy of the normal Cartan connection associated with the induced Lie contact structure on $\mathbb{T}$ reduces to $\mathrm{G}_2$.
 \end{proposition}
 We then  proceed to discuss, more generally, $\mathfrak{so}(3,4)$-Lie contact structures  endowed with a holonomy reduction to $\mathrm{G}_2$. We show the following (see also Theorem \ref{thmHolred}):
 %such a holonomy reduction determines a distinguished rank $2$ distribution $\mathcal{V}$ on an open dense subset  $\widetilde{M}^o$ of the $7$-manifold $\widetilde{M}$. 
  \begin{theorem}\label{thm1.2}
  A holonomy reduction to $\mathrm{G}_2$ determines a distinguished rank $2$ distribution $\mathcal{V}$ on an open dense subset  $\widetilde{M}^o$ of the $7$-manifold $\widetilde{M}$.
  If the curvature of the Cartan connection of the Lie contact structure annihilates the rank $2$ distribution $\mathcal{V}$, then $\mathcal{V}$ is integrable and in a neighbourhood of each point in $\widetilde{M}^o$ one can form a local $5$-dimensional leaf space, which carries an induced  $(2,3,5)$-distribution. Moreover, if $\widetilde{M}^o$ is a proper subset of the $7$-manifold $\widetilde{M}$, then the complement carries an induced parabolic contact structure associated with the Lie algebra $\g_2$.
  \end{theorem}

Our work combines  two approaches: a conceptual one based on theory of parabolic geometries, and explicit calculations in terms of exterior differential systems. The presentation of this article focusses mostly on the parabolic geometry part of our work.
 
 \subsubsection*{Acknowledgements} We thank Gil Bor and Jan Gutt for   reading parts of the paper and useful comments, and Arman Taghavi-Chabert and Vojt\v{e}ch \v{Z}\'adn\'ik for helpful discussions.

\section{Algebraic and geometric background}
A first step to understanding the construction from $(2,3,5)$-distributions to Lie contact structures is to understand the relationship between the homogeneous models of the two structures. 
%In particular, we shall need some background material on the exceptional Lie group $\mathrm{G}_2$.
 Besides presenting the algebra behind the construction, we  further discuss the geometric structures we are interested in.

\subsection{Split octonions and $\mathrm{G}_2$}
The exceptional complex simple Lie algebra $\g_2^{\mathbb{C}}$ has two real forms: the split real form and the compact real form. In the present paper we will be concerned with the split real form $\g_2$ and the corresponding (connected) Lie group $\mathrm{G}_2$ that can be defined  as the automorphism group of the \emph{split octonions} $(\mathbb{O}', \cdot)$. 

An algebra $(\mathcal{A},\cdot)$ with unit $1$
 together with a non-degenerate quadratic form $N$ that is multiplicative in the sense that
$$N(X \cdot Y)=N(X) N(Y)$$
is called a composition algebra. 
There are, up to isomorphism, precisely two $8$-dimensional real composition algebras: the octonions $\mathbb{O}$ and the split octonions $\mathbb{O}'$. The two can be distinguished by the signature of their quadratic forms. The split octonions  are the unique $8$-dimensional real  composition algebra with quadratic form $N:\mathbb{O}'\to\mathbb{R}$ of signature $(4,4)$.

Given a composition algebra, there is a notion of conjugation
$\bar{X}=2\left\langle X, 1\right\rangle 1- X,$
where $\left\langle\; , \;\right\rangle$ denotes the bilinear form determined by $N$ via polarization. 
The space of imaginary split octonions is then defined as 
\[\mathbb{V}=\mathrm{Im}\mathbb{O}'=\{ X\in\mathbb{O}': \bar{X}=- X\}=1^{\perp}.\] 
 Since the unit $1$ has norm one, $\left\langle\, ,\, \right\rangle$ restricts to a bilinear form $H$ of signature $(3,4)$ on $\mathbb{V}$. One can further define a   $3$-form $\Phi\in\Lambda^3\mathbb{V}^*$   as
 \[\Phi(X,Y,Z):=\left\langle X\cdot Y, Z\right\rangle= H(X\times Y,Z),\]
 where $$X\times Y=X\cdot Y +\left\langle X,Y\right\rangle 1$$ denotes the split octonionic cross product on $\mathbb{V}$.
 %= \mathrm{Im}(X\cdot Y)=X\cdot Y -\left\langle X\cdot Y,1\right\rangle 1
 Since an algebra automorphism of a composition algebra preserves the corresponding bilinear form, $\mathrm{G}_2$ preserves all these data. Indeed, it is known that $\mathrm{G}_2$ is precisely the stabilizer of $\Phi$ in $GL(\mathbb{V}),$ 
 and the representation on $\mathbb{V}$ defines an inclusion  $\mathrm{G}_2\hookrightarrow \mathrm{O}(H)=\mathrm{O}(3,4).$
 
\subsection{Explicit matrix presentations of $\g_2$ and $\mathfrak{so}(3,4)$}\label{matrices}
Here we will present an explicit matrix realization of the inclusion
\begin{align}
\mathfrak{g}_2\hookrightarrow \mathfrak{so}(3,4).
\end{align}
 Let $e_1,\cdots,e_7$ be a basis for $\mathbb{V}$ with dual basis $e^1,\cdots,e^7$, i.e. $e^i(e_j)=\delta_{ij}$. Consider the bilinear form 
\begin{align}\label{bilinearform}
H=-2 e^1 e^7-2 e^2 e^6-2 e^3 e^5 -e^4e^4,
\end{align} 
%of signature $(3,4)$, 
defining 
%the following matrix presentation 
 \begin{align}\label{somatrices}
\mathfrak{so}(3,4)=\left\{\begin{pmatrix}
   a^7&-a^3&-a^6&a^{11}&-a^{16}&a^{19}&0\\
    -a^{17}&a^{10}&a^9&a^{15}&-a^{20}&0&-a^{19}\\
    -a^{14}&a^8&a^{13}&a^{18}&0&a^{20}&a^{16}\\
    a^{12}&a^5&a^2&0&-a^{18}&-a^{15}&-a^{11}\\
    -a^{4}&-a^0&0&-a^2&-a^{13}&-a^{9}&a^{6}\\
    a^{1}&0&a^0&-a^5&-a^{8}&-a^{10}& a^{3}\\
    0&-a^1&a^4&-a^{12}&a^{14}&a^{17}&-a^7
    \end{pmatrix},\ a^0,\cdots,a^{20}\in\mathbb{R} \right\}.
\end{align} \\
 Then the subalgebra of $\mathfrak{so}(3,4)$ preserving the $3$-form 
\begin{align}\label{3-form}
\Phi= 2e^1\wedge e^4\wedge e^7+e^1\wedge e^5\wedge e^6+8 e^2\wedge e^3\wedge e^7-2 e^2\wedge e^4\wedge e^6- 2 e^3\wedge e^4 \wedge e^5\end{align}
is the exceptional Lie algebra
%a matrix presentation  of the Lie algebra $\g_2$,
\begin{align}\label{g2matrices}
\g_2=\left\{\begin{pmatrix}
   -q^1-q^4& -2b^6&-12b^5&-2q^5&q^6&-6q^7&0\\
   -\frac{1}{2}b^3&-q^4&6q^2&-6b^5&\frac{1}{2}q^5&0&6q^7\\
   -\frac{1}{12}b^4&\frac{1}{3}q^3&-q^1&b^6&0&-\frac{1}{2}q^5&-q^6\\
   \frac{1}{3}b^0&-\frac{1}{3}b^4&2b^3&0&-b^6&6b^5&2q^5\\
   -b^1&-\frac{2}{3}b^0&0&-2b^3&q^1&-6q^2&12b^5\\
   \frac{1}{6}b^2&0&\frac{2}{3}b^0&\frac{1}{3}b^4&-\frac{1}{3}q^3&q^4&2b^6\\
   0&-\frac{1}{6}b^2&b^1&-\frac{1}{3}b^0&\frac{1}{12}b^4&\frac{1}{2}b^3&q^1+q^4\\
\end{pmatrix},\ b^0,\cdots,q^7\in\mathbb{R}\right\}.
\end{align}
 \subsection{ Parabolic subalgebras
  of $\g_2$ and $\mathfrak{so}(3,4)$
  }\label{parabolics}
 %One possible definition of a parabolic subalgebra of a semisimple Lie algebra $\g$  is the following:
 A subalgebra $\p \subset\g$  of a semisimple Lie algebra $\g$ is a \emph{parabolic subalgebra} if and only if its maximal nilpotent ideal $\p_+$ coincides with the orthogonal complement $\p^{\perp}$  of $\p$ in $\g$ with respect to the
Killing form. 
In particular, this yields an identification $(\g/\p)^*\cong\p_+$.
%The quotient $ \p/\p_+$ is then a reductive Lie algebra. 
A parabolic subalgebra $\p$ determines a filtration 
\begin{align*}
\g=\g^{-k}\supset\cdots\supset\g^0\supset\cdots\supset\g^{k},
\end{align*}
 $[\g^i,\g^j]\subset\g^{i+j}$,\
where $\g^0=\p$, $\g^1=\p^{\perp}$, $\g^i=[\g^1,\g^{i-1}]$ and $\g^{-i}=(\g^i)^{\perp}$ for $i>1$. For a choice of (reductive)  subalgebra $\g_0\subset\p$ isomorphic to $\p/{\p}_+$, called a Levi subalgebra, the filtration splits 
%(into eigenspaces of the so-called grading element) 
which determines  a   grading of the Lie algebra 
\begin{align*}
\g=\g_{-k}\oplus\cdots\oplus\g_0\oplus\cdots\oplus\g_k,
\end{align*}
such that $[\g_i,\g_j ]\subset\g_{i+j}$ and $\g_{-1}$ generates $\g_{-k}\oplus\cdots\oplus\g_{-1}$.
Conversely, given such a grading, 
\begin{align*}
\p:=\g_0\oplus\cdots\oplus\g_k
\end{align*}
defines a parabolic subalgebra, and the filtration can be recovered from the grading as $\g^i=\g_{i}\oplus\cdots\oplus\g_k$.
%Thus, we can equivalently define parabolic subalgebras that way. 
We will now discuss the parabolic subalgebras of $\g=\g_2$ and $\tilde{\g}=\mathfrak{so}(3,4)$ that are relevant for this paper.
 
% \begin{example}
 Consider $\tilde{\g}=\mathfrak{so}(3,4)$ in the matrix presentation \eqref{somatrices}. Let  $\tilde{\p}\subset\mathfrak{so}(3,4)$ be a parabolic subalgebra defined as the stabilizer of a totally null $2$-plane $\mathbb{E}$ with respect to $H$ as in  $\eqref{bilinearform}$.
It preserves the filtration 
 \begin{align}\label{sofilt}\widetilde{\mathbb{V}}^{-1}\supset \widetilde{\mathbb{V}}^{0}\supset\widetilde{\mathbb{V}}^{1},\end{align}
of the standard representation, where $\widetilde{\mathbb{V}}^1=\mathbb{E}$, $\widetilde{\mathbb{V}}^0=(\mathbb{E})^{\perp}$, $\widetilde{\mathbb{V}}^{-1}=\mathbb{V},$
which is related to the filtration 
 $$\tilde{\g}^{-2}\supset\tilde{\g}^{-1}\supset\tilde{\g}^0\supset\tilde{\g}^{1}\supset\tilde{\g}^{2}$$
 determined by $\tilde{\p}$ by  $\tilde{\g}^i=\tilde{\g}\cap L(\widetilde{\mathbb{V}}^j,\widetilde{\mathbb{V}}^{k+i})$. 
 Any two parabolic subalgebras of $\mathfrak{so}(3,4)$ defined as stabilizers of distinct totally null $2$-planes are conjugated to each other by an inner automorphism of $\mathfrak{so}(3,4)$. Hence modulo  conjugation,  they define the same parabolic subalgebra and induced filtration. However, the  parabolics may be different concerning their  position relative to the given subalgebra $\g_2\subset\mathfrak{so}(3,4)$. This observation will be relevant for the purpose of this article.

 Let us first choose the totally null plane $\mathbb{E}'=\mathrm{span}(e_1,e_2)$, where $e_1,\cdots, e_7$ denotes the basis of $\mathbb{V}$ as in Section \ref{matrices}. Then $\tilde{\p}$ consists of upper block triangular matrices. The subalgebra $\tilde{\g}_0\subset\tilde{\p}$  of block diagonal matrices  is a Levi subalgebra  isomorphic to $\mathfrak{gl}(2,\mathbb{R})\oplus\mathfrak{so}(1,2)$. It corresponds to the grading
 \begin{align}\label{sograd1}
  \begin{pmatrix}
      \tilde{\g}_0   & \vline &  \tilde{\g}_1 & \vline &   \tilde{\g}_2 \\
  \hline
   \tilde{\g}_{-1}  & \vline &   \tilde{\g}_0 & \vline &   \tilde{\g}_1 \\
  \hline
   \tilde{\g}_{-2}  & \vline  &  \tilde{\g}_1& \vline &  \tilde{\g}_0\\
  \end{pmatrix}
   \begin{pmatrix}
     \widetilde{\mathbb{V}}_1\\
      \hline
     \widetilde{\mathbb{V}}_0\\
       \hline
     \widetilde{\mathbb{V}}_{-1}\\
     \end{pmatrix},
\end{align}
where the  splitting $\widetilde{\mathbb{V}}_1\oplus\widetilde{\mathbb{V}}_0\oplus{\widetilde{\mathbb{V}}_{-1}}$ of the filtration of $\mathbb{V}$ is given by  $\widetilde{\mathbb{V}}_1=\mathbb{E}'$, $\widetilde{\mathbb{V}}_0=\mathrm{span}(e_3,e_4,e_5)$ and $\widetilde{\mathbb{V}}_{-1}=\mathrm{span}(e_6,e_7)$ ($\widetilde{\mathbb{V}}_1$ is the defining representation for the $\mathfrak{gl}(2,\mathbb{R})$-summand and $\widetilde{\mathbb{V}}_0$  for the $\mathfrak{so}(1,2)$-summand of $\tilde{\g}_0$).
 %, and a splitting of the filtration of $\g$.
% \end{example}
%$$ MATRICES WITH BLOCKS AND STANDARD REP$$

Next we choose a different totally null $2$-plane, $\mathbb{E}=\mathrm{span}(e_2,e_3)$. Looking at the explicit form of the $3$-form $\Phi$ as in \eqref{3-form}, we immediately notice that while ${\mathbb{E}'}\hook\Phi={e_1}\hook {e_2}\hook\Phi=0$, this is not true for the new $2$-plane as  $e_2\hook e_3\hook\Phi= 8 e^7\neq 0$. So inserting the $2$-plane $\mathbb{E}$ into $\Phi$ we obtain the line in $\mathbb{V}^*$ spanned by $e^7$, or using the isomorphism $\mathbb{V}^*\cong\mathbb{V}$ induced by the metric $H$, the line in $\mathbb{V}$ spanned by $e_1$. The grading of $\tilde{\g}$ and corresponding  splitting $\widetilde{\mathbb{V}}_1=\mathrm{span}(e_2,e_3)$,  $\widetilde{\mathbb{V}}_0=\mathrm{span}(e_1,e_4,e_7)$,  $\widetilde{\mathbb{V}}_{-1}=\mathrm{span}(e_5, e_6)$ in this unusual form look as follows:
\begin{align}\label{sograd}
\begin{pmatrix}
  \tilde{\g}_0&\vline&\tilde{\g}_{-1}  &\vline&\tilde{\g}_0   &\vline&\tilde{\g}_1   &\vline&0  \\
    \hline
    \tilde{\g}_{1}&\vline&\tilde{\g}_0  &\vline& \tilde{\g}_1 &\vline&\tilde{\g}_2 &\vline&\tilde{\g}_1  \\
    \hline
   \tilde{\g}_{0} &\vline &\tilde{\g}_{-1}&\vline&\tilde{\g}_0 &\vline&\tilde{\g}_1 &\vline&\tilde{\g}_0  \\
    \hline
    \tilde{\g}_{-1}&\vline &\tilde{\g}_{-2}&\vline&\tilde{\g}_{-1} &\vline& \tilde{\g}_0&\vline&\tilde{\g}_{-1} \\
    \hline
    0&\vline&\tilde{\g}_{-1} &\vline& \tilde{\g}_{0} &\vline&\tilde{ \g}_{1}&\vline&\tilde{ \g}_0& \\
\end{pmatrix}
\begin{pmatrix}
    \widetilde{\mathbb{V}}_0\\
    \hline
    \widetilde{\mathbb{V}}_1\\
    \hline
    \widetilde{\mathbb{V}}_0\\
    \hline
   \widetilde{ \mathbb{V}}_{-1}\\
    \hline
   \widetilde{ \mathbb{V}}_0\\
\end{pmatrix}
\end{align}

Let us now discuss parabolic subalgebras of  $\g_2$; there are, up to conjugation by inner automorphisms of $\g_2$, three of them. Consider $\g=\g_2$ in the matrix representation \eqref{g2matrices}. Let  $\mathfrak{p}\subset \g$ 
 be  the stabilizer of a null-line $\ell$;  we take  the line $ \ell=\mathbb{R} e_1\subset \mathbb{V}$ through the first basis vector $e_1$. 
 (${\mathrm{G}_2}$ acts transitively on  null lines, see e.g. \cite{bryant-exceptional}, and thus different choices lead to conjugated subalgebras). 
 Since $\mathfrak{p}$ preserves $\ell$ and $\Phi$, it also preserves the filtration
\begin{align}\label{filtration}
\mathbb{V}^{-2}\supset \mathbb{V}^{-1}\supset\mathbb{V}^0\supset\mathbb{V}^{1}\supset\mathbb{V}^{2},
\end{align} 
where  $\mathbb{V}^2=\ell=\mathrm{span}(e_1)$, $\mathbb{V}^1=\{ Y\in\mathbb{V}: Y\hook X\hook\Phi =0\ \forall X\in\ell\   \}=\mathrm{span}(e_1,e_2,e_3),$ $\mathbb{V}^0={\mathbb{V}^{-1}}^{\perp}=\mathrm{span}(e_1,e_2,e_3,e_4)$ and $\mathbb{V}^1={\mathbb{V}^{-2}}^{\perp}=\mathrm{span}(e_1,e_2,e_3,e_4,e_5,e_6),$ $\mathbb{V}^{-2}=\mathbb{V}$. Hence $\p$ is an upper block triangular matrix. The filtration of $\mathbb{V}$ and the filtration
$$\g^{-3}\supset\g^{-2}\supset\g^{-1}\supset\g^0\supset\g^1\supset\g^2\supset \g^{3}$$
determined by the parabolic $\p$ are related as $\g^i=\g\cap L(\mathbb{V}^j,\mathbb{V}^{k+i})$.
A choice of Levi subalgebra $\g_0\cong\p/\p^{\perp}\cong \mathfrak{gl}(2,\mathbb{R})$ determines a splitting of the filtration of $\g$, and of $\mathbb{V}$, and vice versa. 
The Levi subalgebra of block diagonal matrices corresponds to the splitting depicted below:
\begin{align}\label{g2grad}
\begin{pmatrix}
  \g_0&\vline&\g_1  &\vline&\g_2   &\vline&\g_3   &\vline&0  \\
    \hline
    \g_{-1}&\vline&\g_0  &\vline& \g_1 &\vline&\g_2 &\vline&\g_3  \\
    \hline
   \g_{-2} &\vline &\g_{-1}&\vline&\g_0 &\vline&\g_1 &\vline&\g_2  \\
    \hline
    \g_{-3}&\vline &\g_{-2}&\vline&\g_{-1} &\vline& \g_0&\vline&\g_1 \\
    \hline
    0&\vline&\g_{-3} &\vline& \g_{-2} &\vline& \g_{-1}&\vline& \g_0& \\
\end{pmatrix}
\begin{pmatrix}
    \mathbb{V}_2\\
    \hline
    \mathbb{V}_1\\
    \hline
    \mathbb{V}_0\\
    \hline
    \mathbb{V}_{-1}\\
    \hline
    \mathbb{V}_{-2}\\
\end{pmatrix}
\end{align}

The other maximal parabolic subalgebra $\bar{\p}$ is the stabilizer in $\g$ of a totally null  $2$-plane    $\mathbb{E}'\subset\mathbb{V}$ such that  $X\hook Y\hook\Phi= 0$ for all $X,Y\in\mathbb{E}'$ (see e.g. \cite{LandsMan}). Let us take  $\mathbb{E}'=\mathrm{span}(e_1,e_2)$. Then, as in the case of the special orthogonal algebra discussed earlier, the parabolic subalgebra is an upper block triangular matrix, and   the induced filtration of $\g$ is of the form
\begin{align}\label{g2bargrad}
\bar{\g}^{-2}\supset\bar{\g}^{-1}\supset\bar{\g}^0\supset\bar{\g}^{1}\supset\bar{\g}^{2},
\end{align}
where $\bar{\g}^{i}=\g\cap\tilde{\g}^{i}$ for the filtration of $\tilde{\g}$ determined by  \eqref{sograd1}.

Finally there is the Borel subalgebra $\mathfrak{b}$ of $\g$; it 
%is the intersection $B=P_1\cap P_2$ and 
is the stabilizer of a filtration $\ell\subset\mathbb{E}'$ of a null line contained in a totally null $2$-plane such that $X\hook Y\hook\Phi= 0$ for all $X,Y\in\mathbb{E}$.

We shall use an analogous notation for the parabolic subgroups appearing in the course of this paper: $\widetilde{P}\subset \mathrm{O}(3,4)$ denotes the stabilizer of a totally null $2$-plane in the standard representation $\mathbb{V},$ $P\subset \mathrm{G}_2$ denotes the stabilizer of a null-line $\ell\subset\mathbb{V}$, $\bar{P}\subset \mathrm{G}_2$ denotes the stabilizer of a totally null $2$-plane  that inserts trivially into the defining $3$-form $\Phi$ for $\mathrm{G}_2$, and $B\subset \mathrm{G}_2$ denotes the stabilizer of a null-line contained in a  null $2$-plane of this type. For reasons that will become clear later, we will call $\widetilde{P}\subset \mathrm{SO}(3,4)$ the \emph{Lie contact parabolic}, $P\subset\mathrm{G}_2$ the \emph{(2,3,5) parabolic} and $\bar{P}\subset\mathrm{G}_2$ the \emph{$\mathrm{G}_2$ contact parabolic}.
  
 \subsection{Parabolic geometries} \label{parabolic}
Here we provide a very brief summary of basic notions from parabolic geometry, mostly to set notation. For a comprehensive introduction to parabolic geometries see \cite{cap-slovak-book}. See also \cite{tanaka} and \cite{yama-simple} for additional information.

A \emph{Cartan geometry} of type $(G,P)$ is given by 
\begin{itemize}
\item a principal bundle $\mathcal{G}\to M$ with structure group $P$,
\item and a Cartan connection  $\omega\in\Omega^1(\mathcal{G},\g)$, i.e., a $P$-equivariant Lie algebra valued $1$-form such that $\omega(u)(\zeta_{X})=X$ for all fundamental vector fields $\zeta_{X}$ and $\omega(u): T_u\mathcal{G}\to\g$ is a linear isomorphism.
\end{itemize}
The curvature of a Cartan connection $\omega$  is the $2$-form in $\Omega^2(\mathcal{G},\g)$ defined as
\[K(\xi,\eta)=d\omega(\xi,\eta)+[\omega(\xi),\omega(\eta)],\]
for $\xi,\eta\in\mathfrak{X}(\mathcal{G})$. It is $P$-equivariant and horizontal, and thus equivalently encoded in the  curvature function $K:\mathcal{G}\to\Lambda^2(\g/\p)^*\otimes\g$  given by
\[K(u)(X,Y)=d\omega\big(\omega^{-1}(u)(X),\omega^{-1}(u)(Y)\big)+[X,Y].\]
It is one of the basic results about Cartan connections that the curvature of a Cartan geometry vanishes, i.e. the geometry is flat, if and only if it is locally equivalent to $G\to G/P$ equipped with the Maurer Cartan form $\omega_{G}$. The latter geometry is referred to as the (homogeneous) model.

A Cartan geometry of type $(G,P)$ is  called a \emph{parabolic geometry} if $\g$ is a semisimple Lie algebra and $P\subset G$ a parabolic subgroup, i.e.,  a closed subgroup with Lie algebra a parabolic subalgebra  $\p\subset \g$. Given a principal bundle $P\hookrightarrow\mathcal{G}\to M$ and Lie algebra $\g$ there are a priori  several choices of Cartan connections $\omega\in\Omega^1(\mathcal{G},\mathfrak{g})$.
In pioneering work Tanaka established the following curvature conditions  that pin down the Cartan connection uniquely: A parabolic geometry is called
\begin{itemize}
\item  \emph{regular} if the curvature  $K$ is of homogeneity $\geq 1$, i.e.,
%\begin{align*}
$K(u)(X,Y)\subset \g^{i+j+1}$
%\end{align*}
for all $X\in\g^i$, $Y\in\g^{j}$ and $u\in\mathcal{G}$,
%. It is called 
\item \emph{normal} if 
%\begin{align*}
$\partial^*\circ K=0,$
%\end{align*}
where 
%\begin{align*}
$\partial^*:\Lambda^2(\g/\p)^*\otimes\g\to(\g/\p)^*\otimes\g$
%\end{align*}
 is  the ($P$-equivariant)  Kostant codifferential. Identifying $(\g/\p)^*=\p_{+}$ via the Killing form, it is the boundary operator  computing the Lie algebra homology $H_*(\p_+,\g),$ given on a decomposable element as 
 \begin{align}\label{KostantCo}
 \partial^*(Z_0\wedge Z_1\otimes A)=Z_0\otimes [Z_1,A]-Z_1\otimes[Z_0,A]-[Z_0,Z_1]\otimes A.
 \end{align}
\end{itemize}
Projecting the curvature $K$ of a regular, normal parabolic geometry to $\mathbb{H}^2:=\mathrm{ker}(\partial^*)/\mathrm{Im}(\partial^*)$ gives the \emph{harmonic curvature} $K_{H}$, which is the fundamental curvature quantity of a regular, normal parabolic geometry.

\subsection{$(2,3,5)$-distributions}
A $(2,3,5)$-distribution $\mathcal{D}\subset TM$ is a rank $2$ distribution on a $5$-manifold that is bracket generating in a minimal number of steps, i.e.
$$[\mathcal{D},[\mathcal{D},\mathcal{D}]]= T M.$$
It follows immediately from the definition that the weak derived  flag $\mathcal{D}\subset[\mathcal{D},\mathcal{D}]\subset T M$ is a sequence of nested  bundles of ranks $2$, $3$ and $5$. 

So, by definition, $(2,3,5)$-distributions are 
in a sense 
opposite to integrable distributions, and
 they are different in character.
While integrable rank $2$-distributions in dimension $5$ are all locally equivalent, $(2,3,5)$-distributions have local invariants. A solution to the
% problem how to determine when two $(2,3,5)$-distributions are locally equivalent 
 local equivalence problem
 was established in Cartan's influential 1910 paper \cite{cartan-cinq}. His work also shows that the symmetry algebra of a $(2,3,5)$ distributions is finite-dimensional; for the most symmetric of these distribution it is the  simple Lie algebra $\g_2$. 

Note that a relationship to $\g_2$ can be seen immediately: If one looks at the symbol algebra of a $(2,3,5)$-distribution, i.e. the associated graded of the derived flag
%$\mathcal{D}_x\oplus [\mathcal{D},\mathcal{D}]_x/\mathcal{D}_x\oplus T _xM/[\mathcal{D},\mathcal{D}]_x$ 
together with the tensorial bracket $\mathcal{L}$ induced by the Lie bracket of vector fields, then this is at each point  a nilpotent Lie algebra isomorphic to the negative part of the grading \eqref{g2grad} of $\g_2$.

Indeed (by Tanaka theory, or Cartan's equivalence method):
\begin{theorem}
There is an equivalence of categories between $(2,3,5)$-distributions and parabolic geometries of type $(\mathrm{G}_2,P)$, where $P\subset \mathrm{G}_2$ is the parabolic subgroup defined as the stabilizer of a null-line in the $7$-dimensional irreducible representation of $\mathrm{G}_2$.
\end{theorem}
Based on the Cartan geometric interpretation of $(2,3,5)$-distributions, a relation  to conformal geometry was observed in \cite{nurowski-metric}:
\begin{theorem}
Every $(2,3,5)$-distribution $\mathcal{D}\subset T M$ determines a conformal class $[g]_{\mathcal{D}}$ of metrics of signature $(2,3)$ on $M$. The distribution $\mathcal{D}$ is totally null with respect to any metric from the conformal class $[g]_{\mathcal{D}}$.
\end{theorem}

\subsection{Lie contact structures}\label{liecontdef}
A contact distribution $\mathcal{H}\subset T M $ on a manifold of dimension $2n+1$ is a co-rank $1$ subbundle such that the  Levi-bracket  
\begin{align*}
\mathcal{L}:\Lambda^2\mathcal{H}\to  T M/\mathcal{H}, \ \mathcal{L}(\xi_x\wedge\eta_x)=[\xi,\eta]_x+\mathcal{H}_x,
\end{align*}
is non degenerate at each point $x\in M$. In other words,  locally, $\mathcal{H}$ is  the kernel of a contact form $\theta$.
%, that is a $1$-form $\theta\in\Omega^1(M)$ such that $\theta\wedge (d\theta)^n\neq 0$. 
Contact distributions do not have local invariants; locally one may always find coordinates $(t,q_i,p_j)$ such that $\theta= \der t-\sum_{i}p_i \der q_i$.

Lie contact structures have been introduced and studied  by Sato and Yamaguchi \cite{Sato-yama1, Sato-yama2}, and  Miyaoka \cite{miyaoka1, miyaoka2}. To state a definition, note that the symbol algebra of a contact distribution $\mathcal{H}$ is at each point isomorphic to the negative part $\tilde{\g}_{-}$ of grading of $\tilde{\g}=\mathfrak{so}(3,4)$ from section \ref{parabolics}. Hence the natural frame bundle for the contact distribution has structure group the grading preserving Lie algebra automorphisms $\mathrm{Aut}_{gr}(\tilde{\g}_{-})$  (which is isomorphic to the conformal symplectic group $\mathrm{CSp}(2n)$). Let $\widetilde{G}_0\cong \mathrm{GL}(2)\times\mathrm{O}(p,q)$ be the subgroup of $\widetilde{P}$ preserving the grading.  A Lie contact structure can be defined as a contact distribution equipped with a reduction of   structure group of the natural frame bundle  with respect to the obvious map   $\widetilde{G}_0\to \mathrm{Aut}_{gr}(\tilde{\g}_{-})$. 

Equivalently, see \cite{cap-slovak-book} and  \cite{vojtech}: 
A Lie contact structure of signature $(p,q)$ on a  manifold $M$ of dimension $2(p+q)+1$ is given by 
\begin{itemize}
\item a contact distribution $\mathcal{H}\subset T M$, 
\item two auxiliary vector bundles, $E\to M$ of rank $2$ and $F\to M$ of rank $p+q$,
and  a bundle metric $b$ of signature $(p,q)$ on $F$, 
\item  an isomorphism $\mathcal{H}\cong E^*\otimes F$ such that   the Levi bracket $\mathcal{L}$  is invariant under the induced action of the orthogonal group $O(b)$ on $\mathcal{H}$. 
\end{itemize}

\begin{theorem} \label{paraLie}There is an equivalence of categories between Lie contact structures of signature $(p,q)$ and regular, normal parabolic geometries of type $(\mathrm{O}(p+2,q+2),\widetilde{P})$, where $\widetilde{P}\subset \mathrm{O}(p+2,q+2)$ is the stabilizer of a totally null $2$-plane.
\end{theorem}

Given a parabolic geometry $(\widetilde{\mathcal{G}}\to\widetilde{M},\widetilde{\omega})$ of type $(\mathrm{O}(p+2,q+2),\widetilde{P})$, vector bundles $E\to M$ and $F\to M$ as in the above description of Lie contact structures are obtained as associated bundles $E=\widetilde{\mathcal{G}}\times_{\widetilde{P}}\mathbb{E}$ and $F=\widetilde{\mathcal{G}}\times_{\widetilde{P}}(\mathbb{E}^{\perp}/\mathbb{E})$, where $\mathbb{E}\subset\mathbb{R}^{p+2,q+2}$ is the totally null $2$-plane stabilized by the parabolic subgroup $\widetilde{P}$.
%$\widetilde{\mathcal{G}}_0=\widetilde{\mathcal{G}}/\widetilde{P}_+$ and $\widetilde{\mathbb{V}}_1$ and $\widetilde{\mathbb{V}}_0$ are the defining representations for the $\mathrm{GL}(2,\mathbb{R})$ factor and the $\mathrm{SO}(1,2)$ factor of $\widetilde{G}_0$, respectively,  see \eqref{sograd1}.

\subsection{$\mathrm{G}_2$-contact structures}\label{G_2_contact}

A $\mathrm{G}_2$-contact structure is defined similarly as a Lie contact structure.
%A $\mathrm{G}_2$-contact structure 
It is given by a contact distribution $\mathcal{H}\subset M$ on a $5$-manifold $M$ together with a reduction of structure group of the natural frame bundle of $\mathcal{H}$ to $\bar{G}_0\cong\mathrm{GL}(2,\mathbb{R})\subset \mathrm{Aut}_{gr}(\bar{\g}_{-})$, where $\bar{\g}_{-}$ denotes the negative part in the grading associated to \eqref{g2bargrad}. Equivalently,  it is a contact distribution $\mathcal{H}$ together with an identification $\mathcal{H}\cong S^3 E$, for some rank $2$ bundle $E\to M$, such that the Levi bracket $\mathcal{L}$ is  invariant under the induced action of $\mathfrak{gl}(E)$, see  \cite{cap-slovak-book}.
Again, by the general theory:
\begin{theorem}
There is an equivalence of categories between $\mathrm{G}_2$-contact structures and parabolic geometries of type $(\mathrm{G}_2,\bar{P})$, where $\bar{P}\subset \mathrm{G}_2$ is the parabolic subgroup defined as the stabilizer of a totally  null $2$-plane that inserts trivially into the defining $3$-form for $\mathrm{G}_2$.
\end{theorem}

\begin{remark}
Yet another description of Lie contact structures and $G_2$-contact structures in terms of a conformal symmetric rank $4$ tensor on the contact distribution  will be provided in sections \ref{underlying_Liecont} and \ref{associated_G_2cont}, respectively.
\end{remark}

\subsection{Relating the models}
\label{models}
  The model for $(2,3,5)$-distributions is the homogeneous space $\mathrm{G}_2/P$ together with its canonical $\mathrm{G}_2$-invariant distribution $\mathcal{D}$. Since $\mathrm{G}_2$ acts transitively on the projective quadric $\mathbb{P}(\mathcal{C})$  of all null-lines with respect to the invariant bilinear form $H$, and $P$ is the stabilizer of such a null-line $\ell$,  we  get an identification
\begin{align*}
\mathrm{G}_2/P\cong \mathbb{P}(\mathcal{C}).
\end{align*}
 The model for the Lie contact structures we are interested in is the homogeneous space
$O(3,4)/\widetilde{P}$ with its canonical left invariant Lie contact structure. Since  $\widetilde{P}$ is the stabilizer of a totally null $2$-plane and  $\mathrm{O}(3,4)$ acts transitively on such $2$-planes, this homogeneous space  can be identified with the $7$-dimensional orthogonal Grassmannian of totally null $2$-planes,
\begin{align*}
\mathrm{O}(3,4)/\widetilde{P}\cong \mathrm{Gr}(2,\mathbb{R}^{3,4}).
\end{align*}
Finally, the homogeneous model for $\mathrm{G}_2$-contact structures is $\mathrm{G}_2/\bar{P}$ with its canonical left-invariant $\mathrm{G}_2$-contact structure.
 
In order to relate the models, recall the $\mathrm{G}_2$-orbit decomposition of the orthogonal Grassmannian $\mathrm{Gr}(2,\mathbb{R}^{3,4})$:
 \begin{proposition}\label{orbits}
 Let $\mathbb{V}$ be a $7$-dimensional vector space with a bilinear form $H$ of signature $(3,4)$, and consider the Grassmannian $\mathrm{Gr}(2,\mathbb{R}^{3,4})$ of totally null $2$-planes in $\mathbb{V}$. Let $\Phi$ be a defining $3$-form for $\mathrm{G}_2\subset \mathrm{O}(3,4)$. 
 Then $\mathrm{Gr}(2,\mathbb{R}^{3,4})$ decomposes into two  $\mathrm{G}_2$-orbits:
 \begin{itemize}
 \item a closed, $5$-dimensional orbit of special $2$-planes $\mathbb{E}'=\mathrm{span}(V,W)$ for which 
\[{\mathbb{E}'}\hook\Phi=V\hook W\hook\Phi= 0,\]
  \item and an open orbit of generic $2$-planes $\mathbb{E}=\mathrm{span}(V,W)$ for which
\[{\mathbb{E}}\hook\Phi=V\hook W\hook\Phi\neq 0.\] 
\end{itemize}
 In the latter  case,  inserting the $2$-plane $\mathbb{E}$ into the $3$-form $\Phi$ defines a line $\ell\subset\mathbb{V}$, which is  null.
The stabilizer  in $\mathrm{G}_2$ of a generic $2$-plane is a $7$-dimensional subgroup $Q=G_0\ltimes \mathrm{exp}(\mathfrak{g}^2)$ isomorphic to $\mathrm{GL}(2,\mathbb{R})\ltimes \mathbb{R}^3$ contained in the parabolic subgroup $P$ that stabilizes the null-line $\ell$.
 \end{proposition}

The $G_2$-orbit decomposition of the orthogonal Grassmannian is known, see e.g. \cite{LandsMan2, pas}. It can  be proven by means of  split octonionic algebra, and we outline the main arguments for a proof in the remark below.

\begin{remark} 
First, let us observe that for any totally null $2$-plane $\mathbb{W}$, $\mathbb{W}\hook\Phi$ is either zero or defines a null line:  Take $\mathbb{V}=\mathrm{Im}\mathbb{O}'$ and $\Phi(X, Y, Z)= H(X\times Y, Z)$. Consider a totally null $2$-plane  $\mathbb{W}=\mathrm{span}(W_1, W_2)\subset \mathbb{V},$ then 
\begin{align*}
W_1\cdot W_2= W_1\times W_2 - \left\langle W_1, W_2\right\rangle 1= W_1\times W_2,
\end{align*}
since $\left\langle W_1, W_2\right\rangle=0$. Hence $\mathbb{W}$ is special if and only if $W_1\cdot W_2=0$ (i.e., it corresponds to a null subalgebra) and generic if and only if $W_1\cdot W_2\neq 0.$ In the latter case 
\begin{align*}
\ell=\mathrm{span}(W_1\cdot W_2)\subset\mathbb{V}
\end{align*}
is a well-defined line  determined by the plane $\mathbb{W}$, and it is null since the quadratic form is multiplicative and both $W_1$ and $W_2$ are null.

Next one needs to prove transitivity of $\mathrm{G_2}$ on generic and special $2$-planes, respectively. This  follows, for instance,    from the fact  that $\mathrm{G}_2$ acts transitively on split octonionic null triples (see \cite{BaezHuerta}, Theorem 13 and Proposition 15): these are ordered triples $X, Y, Z$ of pairwise orthogonal null imaginary split octonions such that $\Phi(X,Y,Z)=\frac{1}{2}$. 
Having established transitivity, the stabilizers can be computed for  arbitrarily chosen $2$-planes in the respective orbits. 

 Let us discuss the stabilizer  $Q$ of a generic null $2$-plane $\mathbb{E}=\mathrm{span}(W_1, W_2)$ in  more detail. First, it preserves the null-line $\ell=\mathrm{span}(W_1\cdot W_2)$ determined by $\mathbb{E}$. Hence, evidently, $Q$ is contained in the parabolic subgroup $P$ stabilizing $\ell$. Next one can show that $ \mathbb{E}\oplus\ell=\mathrm{span}(W_1,W_2,W_1\cdot W_2)$ coincides with 
 \begin{align*}
\{ Z\in\mathbb{V}\ : \ Z\cdot(W_1\cdot W_2)=0\}=\{Z\ :\,Z\hook X\hook \Phi =0\ \forall X\in\ell \}=\mathbb{V}^1,\end{align*} the latter space being the $3$-dimensional filtrant
 in the filtration \eqref{filtration} preserved by the parabolic $P$.
So now we choose a subgroup $G_0\subset Q$, $G_0\cong P/\mathrm{exp}(\p_+)$. Then $P= G_0\ltimes \mathrm{exp}(\mathfrak{g}_1\oplus\mathfrak{g}_2\oplus\mathfrak{g}_3)$, where $\mathrm{exp}(\mathfrak{g}_1)$ acts by (non-zero) maps from $\mathbb{E}$ to $\ell$,
%$\mathrm{exp}(\mathfrak{g}_1)$ maps $\mathbb{W}$ to $\ell$,
 while $\mathrm{exp}(\mathfrak{g}_2\oplus\mathfrak{g}_3)$ acts trivially on $\mathbb{E}$. Hence the subgroup $Q$, which  preserves $\mathbb{E},$  is isomorphic to $G_0\ltimes \mathrm{exp}(\mathfrak{g}_2\oplus\mathfrak{g}_3)$.

\end{remark}

\begin{remark}
Proposition \ref{orbits} shows that the open $\mathrm{G}_2$-orbit in  $\mathrm{O}(3,4)/\tilde{P}$ fibres over $\mathrm{G}_2/P$,
  \[
\begin{tikzcd}
P/Q \arrow[hook]{r} &\arrow{d} \mathrm{G}_2/Q&\\
   &\mathrm{G}_2/P&
\end{tikzcd}
\]
and its five-dimensional boundary is isomorphic to $\mathrm{G}_2/\bar{P}$, i.e., the homogeneous model space for $\mathrm{G}_2$-contact structures.
\end{remark}

\begin{remark} 
%Root diagramm corresponding to  $Q=G_0\ltimes \mathrm{exp}(\mathfrak{g}_2\oplus\mathfrak{g}_3)\subset P$:
In the following root diagram, all black dots correspond to root spaces contained in the standard parabolic $\p$ and the ones with red circles correspond to root spaces contained in the subalgebra $\q=\g_0\oplus\g_2\oplus\g_3\subset\p$:
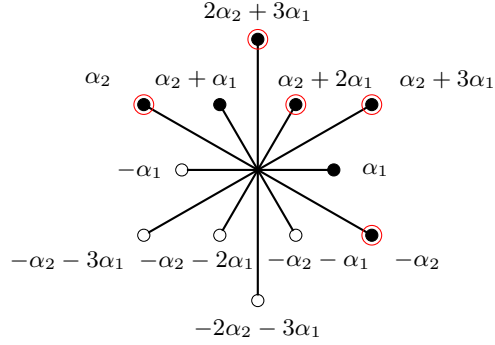
\begin{center}
\begin{tikzpicture}[scale=1]

%\filldraw(0,0) circle(0.15);

\draw[thick] (0,-1.732) -- (0,1.732);
 \filldraw(0,1.732)circle(0.08);
  \draw[red](0,1.732)circle(0.13);
 \draw(0,-1.732)circle(0.08);
 \filldraw[white](0,-1.732)circle(0.06);
 
 \draw[thick](-1.5,-0.866)--(1.5,0.866);
\draw(-1.5,-0.866)circle(0.08);
\filldraw[white](-1.5,-0.866)circle(0.06);
 \filldraw(1.5,0.866)circle(0.08);
  \draw[red](1.5,0.866)circle(0.13);
 
 \draw[thick](1.5,-0.866)--(-1.5,0.866);
\filldraw(-1.5,0.866)circle(0.08);
\draw[red](-1.5,0.866)circle(0.13);
 \filldraw(1.5,-0.866)circle(0.08);
\draw[red](1.5,-0.866)circle(0.13);

\draw[thick](-1,0)--(1,0);
\draw(-1,0)circle(0.08);
\filldraw[white](-1,0)circle(0.06);
 \filldraw(1,0)circle(0.08);
% \draw[red](1,0)circle(0.13);
 
 \filldraw[thick](0.5,-0.866)--(-0.5,0.866);
 \draw(0.5,-0.866)circle(0.08);
  \filldraw[white](0.5,-0.866)circle(0.06);
 \filldraw(-0.5,0.866)circle(0.08);
  %\draw[red](-0.5,0.866)circle(0.13);
 
  \draw[thick](-0.5,-0.866)--(0.5,0.866);
 \draw(-0.5,-0.866)circle(0.08);
 \filldraw[white](-0.5,-0.866)circle(0.06);
 \filldraw(0.5,0.866)circle(0.08);
  \draw[red](0.5,0.866)circle(0.13);

 \draw (2.5,1.2) node {\small {$\alpha_2+3\alpha_1$}};
 \draw (-2.5,-1.2) node {\small{$-\alpha_2-3\alpha_1$}};
 \draw (0.9,1.2) node {\small{$\alpha_2+2\alpha_1$}};
  \draw (-0.8,-1.2) node {\small{$-\alpha_2-2\alpha_1$}};
  %\draw[purple,rotate=-30,yshift=28pt] (-1.2,0.15);
  % rectangle (1.2,-0.15); % g_2
 \draw (1.55,0) node {\small{$ \alpha_1$}};   
  \draw (-1.55,0) node {\small{$- \alpha_1$}};   
\draw(2.1,-1.2) node {\small{$-\alpha_2$}} ;
\draw(-2.1,1.2) node {\small{$\alpha_2$}} ;
 \draw(0.8,-1.2) node {\small{$-\alpha_2-\alpha_1 $}};
  \draw(-0.8,1.2) node {\small{$\alpha_2+\alpha_1 $}};
 \draw (0,-2.1) node {\small{$-2\alpha_2-3\alpha_1$}};
 \draw (0,2.1) node {\small{$2\alpha_2+3\alpha_1$}};

\end{tikzpicture}
\end{center}
\end{remark}

\section{From $(2,3,5)$-distributions to Lie contact structures}\label{secContruction}
%\section{A $7$-manifold with interesting geometric structure}
In this section we  present a natural geometric construction of a  $7$-dimensional twistor bundle over a $5$-manifold equipped with a $(2,3,5)$-distribution, and we investigate the induced geometric structure on the twistor bundle. In particular, we will prove Theorem \ref{thm1.1}.

\subsection{The $(2,3,5)$ twistor bundle}\label{twistor-bundle}
Let $\mathcal{D}$ be a $(2,3,5)$ distribution on a $5$-manifold $M$ with derived flag $\mathcal{D}\subset[\mathcal{D},\mathcal{D}]\subset TM$ and conformal class $[g]_{\mathcal{D}}.$ Then we can form the  bundle 
\begin{align*}
\pi:\mathbb{P}([\mathcal{D},\mathcal{D}])=\cup_{x\in M}\{ \ell_x\subset[\mathcal{D},\mathcal{D}]_x\}\to M
\end{align*} 
of all lines contained in the rank $3$-distribution.
% $[\mathcal{D},\mathcal{D}]$, i.e.\[\mathcal{P}([\mathcal{D},\mathcal{D}])=\cup_{x\in M}\{ \ell_x\subset[\mathcal{D},\mathcal{D}]_x\}.\]  
The $7$-dimensional manifold  $\mathbb{P}([\mathcal{D},\mathcal{D}])$ decomposes as 
$\mathbb{P}([\mathcal{D},\mathcal{D}])=\mathbb{P}(\mathcal{D})\cup \mathbb{T}$ into the the $6$-dimensional subset $\mathbb{P}(\mathcal{D})$ of all lines contained in $\mathcal{D},$ and the open subset $\mathbb{T}$ of all lines in $[\mathcal{D},\mathcal{D}]$ transversal to $\mathcal{D}$. The  space $\mathbb{P}(\mathcal{D})$ has an interesting induced geometry, but here we are interested in the complement:

\begin{definition}\label{twistordef}
We call 
%the $7$-manifold
$$\mathbb{T}= \mathbb{P}([\mathcal{D},\mathcal{D}]) \setminus  \mathbb{P}(\mathcal{D})= \cup_{x\in M}\{ \ell_x\subset[\mathcal{D},\mathcal{D}]_x: \ell_x\not\subset\mathcal{D}\}$$
the twistor bundle of the $(2,3,5)$-distribution $\mathcal{D}$.
\end{definition}
\vspace{.05cm}

\begin{remark}~\\ \vspace{-.5cm}
\begin{itemize}
\item Since $\mathcal{D}$ is totally null with respect to $[g]_{\mathcal{D}}$, we can equivalently describe  $\mathbb{T}$ as the space of all non-null lines  contained in $[\mathcal{D},\mathcal{D}]. $
\item Via the conformal structure, we can identify  $\mathcal{P}(T M )$ with $\mathcal{P}(T^*M)$. Under this identification, $\mathbb{T}$ corresponds to the space of lines in the cotangent space that annihilate $\mathcal{D}$ but do not annihilate $[\mathcal{D},\mathcal{D}]$:
$$\mathbb{T}= \mathbb{P}({\mathcal{D}}^{\perp}) \setminus  \mathbb{P}({[\mathcal{D},\mathcal{D}]}^{\perp})=\cup_{x\in M}\{ \ell_x\subset {\mathcal{D}_x}^{\perp}: \ell_x\not\subset{[\mathcal{D},\mathcal{D}]_x}^{\perp}\}\subset \mathbb{P}(T^*M).$$
\end{itemize}
\end{remark}
Among the geometric structures that are naturally present on the twistor bundle $\mathbb{T}$ we are particularly interested in the  rank $6$ sub-bundle 
$$\mathcal{H}=\cup_{\ell\in\mathbb{T}}\{ \xi\in T_{\ell}\mathbb{T} : \pi_{*}(\xi)\in\ell^{\perp} \}, $$ where the orthogonal complement $\ell^{\perp}$ is taken with respect to the conformal class $[g]_{\mathcal{D}}$ on $M$.
Alternatively, if we realize $\mathbb{T}$  inside $\mathbb{P}(T^*M)$, then $\mathcal{H}$ is precisely the intersection of the canonical contact distribution on $\mathbb{P}(T^*M)$ with $T\mathbb{T}$.
Now it is not difficult to see that $\mathcal{H}\subset T\mathbb{T}$ defines   a contact structure on $\mathbb{T}$. In the following we will show more, we will prove that $\mathbb{T}$ has a naturally induced Lie contact structure of signature $(1,2)$.

\subsection{The induced  Lie contact structure }
We shall prove Theorem \ref{thm1.1} using the  descriptions of $(2,3,5)$-distributions and Lie contact structures, respectively, in terms  of Cartan geometries.
There is a very general functorial construction that assigns to a Cartan geometry of some type $(G,P)$ over a manifold $M$ a Cartan geometry of a different type $(\widetilde{G},\widetilde{P})$ over a manifold $\widetilde{M}$. In the context of parabolic geometries these constructions are referred to as Fefferman-type constructions, see  \cite{cap-constructions, cap-slovak-book}. We briefly recall the general principles.
 
Suppose we have an inclusion $i:G\hookrightarrow\widetilde{G}$ of Lie groups, and subgroups $P$ and $\widetilde{P}$ such that the $G$-orbit of $o=e\widetilde{P}\in\widetilde{G}/\widetilde{P}$ is open and  $Q:=i^{-1}(\widetilde{P})\subset G$ is contained in $P$.
Then the construction proceeds in two steps.  Let $(\mathcal{G}\to M,\omega)$ be a Cartan geometry of type $(G,P)$. 
%Denote by $Q$ the subgroup $Q:=i^{-1}(\widetilde{P})\subset P$,
Now form the so-called \emph{correspondence space}
\begin{align}\label{corres}
\widetilde{M}=\mathcal{G}/Q,
\end{align} and regard $\omega\in \Omega^1(\mathcal{G},\mathfrak{g})$ as a Cartan connection on the $Q$-principal bundle $\mathcal{G}\to \widetilde{M}$. Then $(\mathcal{G}\to \widetilde{M},\omega)$ is automatically a Cartan geometry of type $(G,Q)$.
In a second step, extend the structure group $$\widetilde{\mathcal{G}}:=\mathcal{G}\times_{Q}\widetilde{P},$$ such that $\widetilde{\mathcal{G}}\to\widetilde{M}$ is now a $\widetilde{P}$-principal bundle over $\widetilde{M}$. Let $j:\mathcal{G}\to\widetilde{\mathcal{G}}$ be corresponding bundle inclusion. Since the $G$-orbit of $e\widetilde{P}$ in $\widetilde{G}/\widetilde{P}$ is open, there is a unique extension of $\omega$ to a Cartan connection $\widetilde{\omega}\in\Omega^1(\widetilde{\mathcal{G}},\tilde{\g})$ such that
%\begin{align*}
$j^*\widetilde{\omega}=\omega,$
%\end{align*} 
see \cite{cap-slovak-book}. Thus, we obtain a Cartan geometry $(\widetilde{\mathcal{G}}\to \widetilde{M},\widetilde{\omega})$  of type $(\widetilde{G},\widetilde{P})$.

The curvature functions $\widetilde{K}:\widetilde{\mathcal{G}}\to\Lambda^2(\tilde{\g}/\tilde{\p})^*\otimes\tilde{\g}$ and $K:\mathcal{G}\to\Lambda^2(\g/\p)^*\otimes\g$ of the respective Cartan geometries are related as
\begin{align*}
\widetilde{K}\circ j=(\Lambda^2\varphi\otimes i')\circ K,
\end{align*}
where $i':\mathfrak{g}\to \tilde{\mathfrak{g}}$ is the derivative of the Lie group homomorphism $i$ and  $\varphi:(\mathfrak{g}/\mathfrak{p})^*\to (\tilde{\mathfrak{g}}/\tilde{\mathfrak{p}})^*$ is the dual map to the projection $\tilde{\mathfrak{g}}/\tilde{\mathfrak{p}}\cong \mathfrak{g}/\mathfrak{q}\to\mathfrak{g}/\mathfrak{p}$.

Now we specialize to our groups. We take $G=\mathrm{G}_2$ and $\widetilde{G}=\mathrm{O}(3,4)$, so in particular we have an inclusion $i:G\hookrightarrow \widetilde{G}$.  Then we take $P$ to be the parabolic subgroup in $\mathrm{G}_2$ that stabilizes a null-line $\ell\in\mathbb{R}^7$, and $\widetilde{P}$ to be the stabilizer in $\mathrm{O}(3,4)$ of a generic null $2$-plane $\mathbb{E}\subset\mathbb{R}^7$ such that the null-line determined by $\mathbb{E}$ is $\ell$, i.e. $\iota_{\mathbb{E}}\Phi=\ell$. By Proposition \ref{orbits} this  means that the $G$-orbit of $o=e\widetilde{P}\in\widetilde{G}/\widetilde{P}$ is open and the subgroup $Q=i^{-1}(\widetilde{P})$ is contained in the parabolic  $P$. 

Given a $(2,3,5)$ distribution $\mathcal{D}$ with its canonical Cartan geometry $(\mathcal{G}\to M,\omega)$ of type $(G,P)$, it then follows immediately from the general considerations outlined above that there is a naturally associated Cartan geometry $(\widetilde{\mathcal{G}}\to \widetilde{M},\widetilde{\omega})$ of type $(\widetilde{G},\widetilde{P}).$ It remains to show that this Cartan geometry (which is of the right type) determines a Lie contact structure on $\widetilde{M}$. This is the case provided the curvature $\widetilde{K}$ is regular, i.e. $\widetilde{K}(u)(\tilde{\g}^i,\tilde{\g}^j)\subset \tilde{\g}^{i+j+1}$ at any point $u\in\widetilde{\mathcal{G}}$.

 \begin{remark}
To understand the geometric meaning of the regularity condition, note that the Cartan connection $\widetilde{\omega}$ determines an isomorphism
$$T\widetilde{M}\cong\widetilde{\mathcal{G}}\times_{\widetilde{P}}\tilde{\g}/\tilde{\p}$$
and via this isomorphism  the $\widetilde{P}$-invariant subspace $\tilde{\g}^{-1}/\tilde{\p}\subset \tilde{\g}/\tilde{\p}$ gives rise to a rank $6$-subbundle 
 $$\mathcal{H}\cong\widetilde{\mathcal{G}}\times_{\widetilde{P}}\tilde{\g}^{-1}/\tilde{\p}.$$
Now the regularity condition ensures that the bundle map $\mathcal{L}$ on the graded bundle $\mathrm{gr}(T\widetilde{M})=\mathcal{H}\oplus T\widetilde{M}/\mathcal{H}$  induced by the Lie bracket of vector fields coincides with the one induced by the algebraic Lie bracket on $\mathrm{gr}(\tilde{\g}/\tilde{\p})=\tilde{\g}_{-2}\oplus\tilde{\mathfrak{g}}_{-1}$. Inspecting the Lie bracket on $\tilde{\g}_{-2}\oplus\tilde{\mathfrak{g}}_{-1}$ immediately shows that this implies that $\mathcal{L}:\Lambda^2\mathcal{H}\to T\widetilde{M}/\mathcal{H}$ is non-degenerate, i.e., $\mathcal{H}$ is a contact distribution. To see that one indeed gets an induced Lie contact structure, 
note that as a $\widetilde{P}$-representation $\tilde{\g}^{-1}/\tilde{\p}={\mathbb{E}}^*\otimes\mathbb{E}^{\perp}/\mathbb{E}$. See also \cite{cap-slovak-book}.
\end{remark}

\begin{proposition}\label{regular}
Suppose $(\mathcal{G}\to M,\omega)$ is a regular and normal parabolic geometry of type $(G,P)$, then the induced parabolic geometry $(\widetilde{\mathcal{G}}\to \widetilde{M},\widetilde{\omega})$  of type $(\widetilde{G},\widetilde{P})$ is regular.
 In particular it determines a Lie contact structure on the manifold $\widetilde{M}=\mathcal{G}/Q$.
\end{proposition}
\begin{proof}
It is known, see  \cite{nurowski-metric} or Theorem \ref{g2_gs}, that  the regular, normal Cartan geometry $(\mathcal{G}\to M,\omega)$ associated with a $(2,3,5)$ distribution  is torsion-free, i.e. the curvature function $K$ takes values in $\Lambda^2(\g/\p)^*\otimes\p$. 
%Hence the curvature function of the $(G,Q)$ Cartan geometry $(\mathcal{G}\to \tilde{M},\omega)$ takes values in $\Lambda^2(\g/\q)^*\otimes\p$. 
Via the inclusion $\g\hookrightarrow\tilde{\g},$ the parabolic $\p$ is contained in the $\widetilde{P}$-module $\tilde{\g}^{-1}$, and so the curvature function $\widetilde{K}$ of the  Cartan geometry $(\widetilde{\mathcal{G}}\to \widetilde{M},\widetilde{\omega})$ takes values in $\Lambda^2(\tilde{\g}/\tilde{\p})^*\otimes\tilde{\g}^{-1}.$ This implies that the curvature $\widetilde{K}$ is of homogeneity $\geq 1$, i.e. the geometry is regular. 

\end{proof}

Next we show that $\widetilde{M}$ is the twistor bundle $\mathbb{T}$ as introduced in Definition \ref{twistordef}.
\begin{proposition}
The manifold $\widetilde{M}=\mathcal{G}/Q$ can be naturally identified with the twistor bundle $\mathbb{T}=\cup_{x\in M}\{ \ell_x\in[\mathcal{D},\mathcal{D}]_x: \ell_x\notin\mathcal{D}\}$ of all lines in $[\mathcal{D},\mathcal{D}]$ transversal to $\mathcal{D}$.
\end{proposition}
\begin{proof}
By definition,
\begin{align*}
\widetilde{M}=\mathcal{G}/Q=\mathcal{G}\times_{P}P/Q.
\end{align*}
Let  $\g^{-1}/\p\subset \g^{-2}/\p\subset\g^{-3}/\p$ be the $P$-invariant filtration on $\g/\p$. To prove the proposition it remains to identify the homogeneous space $P/Q$ with the set of lines in $\g^{-2}/\p$ that are not contained in $\g^{-1}/\p$. 

We have noticed in the proof of Proposition \ref{orbits} that $Q=G_0\ltimes \mathrm{exp}(\mathfrak{g}_2\oplus\mathfrak{g}_3)$ for some subgroup $G_0\cong P/P_+$ and corresponding $G_0$-invariant grading $\g_{-3}\oplus\g_{-2}\oplus\g_{-1}\oplus\g_0\oplus\g_1\oplus\g_2\oplus\g_3$. Now $\mathrm{exp}(\mathfrak{g}_2\oplus\mathfrak{g}_3)$ acts trivially on $\mathfrak{g}^{-2}/\mathfrak{p}$ and $G_0$ preserves the 
 line $\ell=(\g_{-2}+\mathfrak{p})/\mathfrak{p}$ (and acts non-trivially on it).
%decomposition $\mathfrak{g}^{-2}/\mathfrak{p}=\ell\oplus\g^{-1}/\p$, where  $\ell$ denotes the line $\ell=(\g_{-2}+\mathfrak{p})/\mathfrak{p}$. 
On the other hand, the action identifies $\mathrm{exp}(\g_1)$ with the space of linear maps from $\ell$ to $\g^{-1}/\p$. It follows that the $P$-action  is transitive on lines in $\g^{-2}/\mathfrak{p}$ not contained in $\g^{-1}/\mathfrak{p}$ and the stabilizer of $\ell$ as above is the subgroup $Q$.

\end{proof}
In particular, we have proven Theorem \ref{thm1.1}.

\begin{remark} In \cite{cap-slovak-book} a construction from conformal structures to Lie contact structures is presented, which generalizes the work of Miyaoka, Sato and Yamaguchi  \cite{miyaoka1,miyaoka2, Sato-yama2}.
Note that the Lie contact structure constructed here is different from the Lie contact structure associated with the conformal structure $[g_{\mathcal{D}}]$ following their construction. The latter one lives on a $9$-dimensional manifold, ours on a $7$-manifold.
\end{remark}

\subsection{Additional structure on the twistor bundle} One immediately observes that the Lie contact structures obtained from $(2,3,5)$-distributions are special. In particular, besides $\mathcal{H}$,
there are several other naturally defined distributions on $\mathbb{T}$.  First there is the vertical bundle 
%$\mathcal{V}\subset T \widetilde{M}$ 
$$\mathcal{V}=\cup_{\ell\in \mathbb{T}}\{ \xi\in T_{\ell}\mathbb{T}: \pi_{*}(\xi)=0 \}$$
for the projection $\pi:\mathbb{T}\to M$, which has rank $2$. Then there are 
the lifts of  $\mathcal{D}$ and $[\mathcal{D},\mathcal{D}]$,
$$\widetilde{\mathcal{D}}=\cup_{\ell\in \mathbb{T}}\{ \xi\in T_{\ell}\mathbb{T} : \pi_{*}(\xi)\in\mathcal{D} \},$$
and 
$$\widetilde{[\mathcal{D},\mathcal{D}]}=\cup_{\ell\in \mathbb{T}}\{ \xi\in T_{\ell}\mathbb{T}: \pi_{*}(\xi)\in[\mathcal{D},\mathcal{D}] \},$$
 which are bundles of of ranks  $4$ and $5$, respectively.
Finally, there is  a rank $3$ distribution 
$$\mathcal{S}=\cup_{\ell\in \mathbb{T}}\{ \xi\in T_{\ell}\mathbb{T}: \pi_{*}(\xi)\in\ell \}, $$
called the prolongation of $\mathcal{D}$.

These distributions can be  understood as follows: Since  $(\widetilde{\mathcal{G}}\to\widetilde{M},\widetilde{\omega})$ arises as the extension of a Cartan geometry $(\mathcal{G}\to\widetilde{M},\omega)$ of type $(G,Q)$, we have an isomorphism
\[T\widetilde{M}\cong\mathcal{G}\times_{Q}\g/\q\]
via the Cartan connection $\omega$. In particular,  every $Q$-invariant subspace of $\mathfrak{g}/\mathfrak{q}$ corresponds to a natural subbundle of $T\widetilde{M}$.  
The vertical bundle $\mathcal{V}$ corresponds to $\p/\q,$ the rank $3$ bundle $\mathcal{S}$ corresponds to $(\g_{-2}+\p)/\q,$ the contact subbundle $\mathcal{H}$ corresponds to $(\g_{-3}\oplus\g_{-1}\oplus \p)/\q,$ and $\widetilde{\mathcal{D}}$ and $\widetilde{[\mathcal{D},\mathcal{D}]}$ correspond to  $\g^{-1}/\q$ and $\g^{-2}/\q$, respectively. 
We can visualize these $Q$-submodules using the root diagram for $\mathrm{G}_2$.
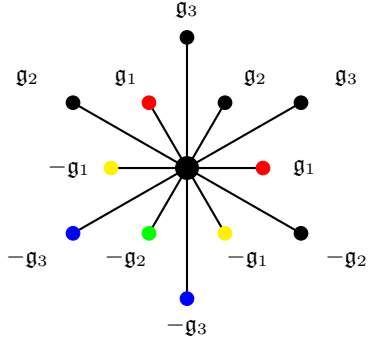
\begin{center}

\begin{tikzpicture}[scale=1]

\filldraw(0,0) circle(0.15);

\draw[thick] (0,-1.732) -- (0,1.732);
 \filldraw(0,1.732)circle(0.09);
 \filldraw[blue](0,-1.732)circle(0.09);
 
 \draw[thick](-1.5,-0.866)--(1.5,0.866);
\filldraw[blue](-1.5,-0.866)circle(0.09);
 \filldraw(1.5,0.866)circle(0.09);
 
 \draw[thick](1.5,-0.866)--(-1.5,0.866);
\filldraw(-1.5,0.866)circle(0.09);
 \filldraw(1.5,-0.866)circle(0.09);

\draw[thick](-1,0)--(1,0);
\filldraw[yellow](-1,0)circle(0.09);
 \filldraw[red](1,0)circle(0.09);
 
 \draw[thick](0.5,-0.866)--(-0.5,0.866);
 \filldraw[yellow](0.5,-0.866)circle(0.09);
 \filldraw[red](-0.5,0.866)circle(0.09);
 
  \draw[thick](-0.5,-0.866)--(0.5,0.866);
 \filldraw[green](-0.5,-0.866)circle(0.09);
 \filldraw(0.5,0.866)circle(0.09);

 \draw (2.1,1.2) node {\small {$\g_3$}};
 \draw (-2.1,-1.2) node {\small{$-\g_3$}};
 \draw (0.9,1.2) node {\small{$\g_2$}};
  \draw (-0.8,-1.2) node {\small{$-\g_2$}};
  %\draw[purple,rotate=-30,yshift=28pt] (-1.2,0.15);
  % rectangle (1.2,-0.15); % g_2
 \draw (1.55,0) node {\small{$\g_1$}};   
  \draw (-1.55,0) node {\small{$- \g_1$}};   
\draw(2.1,-1.2) node {\small{$-\g_2$}} ;
\draw(-2.1,1.2) node {\small{$\g_2$}} ;
 \draw(0.8,-1.2) node {\small{$-\g_1 $}};
  \draw(-0.8,1.2) node {\small{$\g_1 $}};
 \draw (0,-2.1) node {\small{$-\g_3$}};
 \draw (0,2.1) node {\small{$\g_3$}};
\end{tikzpicture}
\end{center}

\section{The exterior differential system and examples }
%{The construction via the corresponding exterior differential system}
\label{secExamples}

Here we present a slightly different viewpoint on the construction of Lie contact structures from $(2,3,5)$ distributions, complementing the picture from the previous section. 
First we  present the structure equations, or exterior differential system (EDS),  for $(2,3,5)$ distributions. Then we  show how they can be applied to (locally) construct  the induced Lie contact structures in terms of a conformal symmetric  rank $4$ tensor on the contact distribution $\mathcal{H}$. This viewpoint has the advantage that it leads to explicit formulae and enables us, for instance, to solve the symmetry equations for a given structure.
 This is carried out  for a special class of distributions  parametrized by functions $F(q)=\frac{q^k}{k(k-1)}$.
 % $F=h(q)$   of a single variable $q$.
 %, which enables us the determine the symmetry algebras for the constructed geometries.

\subsection{The EDS for a $(2,3,5)$ distribution}
The EDS for a generic $(2,3,5)$ distribution was first introduced by Cartan in \cite{cartan-cinq}, and was then modified in \cite{nurowski-metric} to get a form adapted to the corresponding (reduced to $\mathfrak{g}_2$) normal conformal Cartan connection. Here we have rewritten the system from \cite{nurowski-metric} changing the notation to be more suitable to the contact structures we consider in this paper. The changes in notations with respect to \cite{nurowski-metric} are as follows: 
 $$\begin{tabular}{|c|c|} \hline
 1-forms in \cite{nurowski-metric} & the respective 1-forms in this paper \\ \hline\vspace{-.3cm}
 \phantom{}&\phantom{}\\
 $\theta^1, \theta^2, \theta^3, \theta^4, \theta^5$& $\theta^1, \theta^2, \theta^0, \theta^3, \theta^4$\\
$\Omega_5, \Omega_6$&$3\theta^6, 3\theta^5$\\
$\Omega_7, \Omega_8, \Omega_9$&$\Omega_5, \Omega_6, \Omega_7$\\\hline
  \end{tabular}$$
\begin{theorem}\label{g2_gs}
A $(2,3,5)$-distribution $\mathcal{D}$ on a $5$-manifold $M$ uniquely defines a $14$-dimensional bundle $P\to \mathcal{G}\to M$ together with a rigid coframe $(\theta^0,\theta^1,\theta^2,\theta^3,\theta^4,$ $\theta^5,\theta^6,\Omega^1,\Omega^2,\Omega^3,\Omega^4,\Omega^5,\Omega^6,\Omega^7)$  on it satisfying the following exterior differential system (EDS):
$$
%{\small
\begin{aligned}
\der\theta^0&=\theta^0\dz(\Omega_1+\Omega_4)+3\theta^1\dz\theta^6+3\theta^2\dz\theta^5+\theta^3\dz\theta^4,\\
\der\theta^1&=\theta^0\dz\theta^3+\theta^1\dz(2\Omega_1+\Omega_4)+\theta^2\dz\Omega_2,\\
\der\theta^2&=\theta^0\dz\theta^4+\theta^1\dz\Omega_3+\theta^2\dz(\Omega_1+2\Omega_4),\\
\der\theta^3&=4\theta^0\dz\theta^5+\theta^1\dz\Omega_5+\theta^3\dz\Omega_1+\theta^4\dz\Omega_2,\\
 \der\theta^4&=-4\theta^0\dz\theta^6+\theta^2\dz\Omega_5+\theta^3\dz\Omega_3+\theta^4\dz\Omega_4,\\
% \end{aligned}} $$
% $${\small  \begin{aligned}
  \der\Omega_1&=-\Omega_2\dz\Omega_3-\tfrac13\Omega_5\dz\theta^0-\Omega_6\dz\theta^1-2\theta^3\dz\theta^6+\theta^4\dz\theta^5\\
  &-b_2\theta^0\dz\theta^1-b_3\theta^0\dz\theta^2+\tfrac38 c_2\theta^1\dz\theta^2+a_2\theta^1\dz\theta^3\\
 & +a_3(\theta^1\dz\theta^4+\theta^2\dz\theta^3)+a_4\theta^2\dz\theta^4,\\
    \der\Omega_2&=-\Omega_1\dz\Omega_2-\Omega_2\dz\Omega_4-\Omega_7\dz\theta^1-3\theta^3\dz\theta^5\\ 
  &-b_3\theta^0\dz\theta^1-b_4\theta^0\dz\theta^2+\tfrac38 c_3\theta^1\dz\theta^2+a_3\theta^1\dz\theta^3\\
&  +a_4(\theta^1\dz\theta^4+\theta^2\dz\theta^3)+a_5\theta^2\dz\theta^4,\\
 %\end{aligned} $$
% $$  \begin{aligned}
\der\Omega_3&=\Omega_1\dz\Omega_3+\Omega_3\dz\Omega_4-\Omega_6\dz\theta^2-3\theta^4\dz\theta^6\\
 &+b_1\theta^0\dz\theta^1+b_2\theta^0\dz\theta^2-\tfrac38 c_1\theta^1\dz\theta^2-a_1\theta^1\dz\theta^3\\
 &-a_2(\theta^1\dz\theta^4+\theta^2\dz\theta^3)-a_3\theta^2\dz\theta^4,\\
  \end{aligned} $$
 $$  \begin{aligned}
   \der\Omega_4&=\Omega_2\dz\Omega_3-\tfrac13\Omega_5\dz\theta^0-\Omega_7\dz\theta^2+\theta^3\dz\theta^6-2\theta^4\dz\theta^5\\
 &+b_2\theta^0\dz\theta^1+b_3\theta^0\dz\theta^2-\tfrac38 c_2\theta^1\dz\theta^2-a_2\theta^1\dz\theta^3\\
 &-a_3(\theta^1\dz\theta^4+\theta^2\dz\theta^3)-a_4\theta^2\dz\theta^4,\\
  \der\theta^5&=\Omega_2\dz\theta^6+\Omega_4\dz\theta^5-\tfrac13\Omega_5\dz\theta^3-\tfrac13\Omega_7\dz\theta^0\\
 &-\tfrac14 c_2\theta^0\dz\theta^1-\tfrac14 c_3\theta^0\dz\theta^2+e_1\theta^1\dz\theta^2\\
 &+\tfrac14 b_2\theta^1\dz\theta^3+\tfrac14 b_3(\theta^1\dz\theta^4+\theta^2\dz\theta^3)+\tfrac14 b_4\theta^2\dz\theta^4,\\
\der\theta^6&=\Omega_1\dz\theta^6+\Omega_3\dz\theta^5+\tfrac13\Omega_5\dz\theta^4-\tfrac13\Omega_6\dz\theta^0\\
 &-\tfrac14 c_1\theta^0\dz\theta^1-\tfrac14 c_2\theta^0\dz\theta^2+e_2\theta^1\dz\theta^2\\
 &+\tfrac14 b_1\theta^1\dz\theta^3+\tfrac14 b_2(\theta^1\dz\theta^4+\theta^2\dz\theta^3)+\tfrac14 b_3\theta^2\dz\theta^4,\\
  \der\Omega_5&=\Omega_1\dz\Omega_5+\Omega_4\dz\Omega_5-\Omega_6\dz\theta^3-\Omega_7\dz\theta^4-12\theta^5\dz\theta^6\\
 &+4 e_2\theta^0\dz\theta^1+4 e_1\theta^0\dz\theta^2+f\theta^1\dz\theta^2-\tfrac38 c_1\theta^1\dz\theta^3\\
 &-\tfrac38 c_2(\theta^1\dz\theta^4+\theta^2\dz\theta^3)-\tfrac38 c_3\theta^2\dz\theta^4,\\
%   \end{aligned}
%   %} 
%   $$
% $$
% %{\small 
%  \begin{aligned}
\der\Omega_6&=2\Omega_1\dz\Omega_6+\Omega_3\dz\Omega_7+\Omega_4\dz\Omega_6-3\Omega_5\dz\theta^6\\
 &-p_1\theta^0\dz\theta^1-p_2\theta^0\dz\theta^2+q_1\theta^1\dz\theta^2+h_1\theta^1\dz\theta^3\\
 &+h_2(\theta^1\dz\theta^4+\theta^2\dz\theta^3)+h_3\theta^2\dz\theta^4,\\
\der\Omega_7&=\Omega_1\dz\Omega_7+\Omega_2\dz\Omega_6+2\Omega_4\dz\Omega_7-3\Omega_5\dz\theta^5\\
 &-\tfrac13(2f+3p_2)\theta^0\dz\theta^1-p_3\theta^0\dz\theta^2+q_2\theta^1\dz\theta^2+(h_2-e_2)\theta^1\dz\theta^3\\
 &+(h_3-e_1)(\theta^1\dz\theta^4+\theta^2\dz\theta^3)+h_4\theta^2\dz\theta^4.
\end{aligned}
%}
$$
The functions $a_1,a_2,a_3,a_4,a_5,   b_1,b_2,b_3,b_4,  c_1,c_2,c_3,  e_1,e_2,  f,  q_1,q_2,  p_1,p_2,p_3,  h_1,h_2,h_3,h_4$ appearing in the EDS can be understood as the curvature coefficients of the normal Cartan connection $\omega\in\Omega^1(\mathcal{G}, \mathfrak{g}_2)$ associated with the distribution $\mathcal{D}$. In terms of the rigid coframe the Cartan normal connection $\omega$ reads
\be
\omega=\begin{pmatrix}
   -\Omega^1-\Omega^4& -2\theta^6&-12\theta^5&-2\Omega^5&\Omega^6&-6\Omega^7&0\\
   -\frac{1}{2}\theta^3&-\Omega^4&6\Omega^2&-6\theta^5&\frac{1}{2}\Omega^5&0&6\Omega^7\\
   -\frac{1}{12}\theta^4&\frac{1}{3}\Omega^3&-\Omega^1&\theta^6&0&-\frac{1}{2}\Omega^5&-\Omega^6\\
   \frac{1}{3}\theta^0&-\frac{1}{3}\theta^4&2\theta^3&0&-\theta^6&6\theta^5&2\Omega^5\\
   -\theta^1&-\frac{2}{3}\theta^0&0&-2\theta^3&\Omega^1&-6\Omega^2&12\theta^5\\
   \frac{1}{6}\theta^2&0&\frac{2}{3}\theta^0&\frac{1}{3}\theta^4&-\frac{1}{3}\Omega^3&\Omega^4&2\theta^6\\
   0&-\frac{1}{6}\theta^2&\theta^1&-\frac{1}{3}\theta^0&\frac{1}{12}\theta^4&\frac{1}{2}\theta^3&\Omega^1+\Omega^4\\
\end{pmatrix}.
\label{carcon235}
\ee
The curvature $K$ of the connection $\omega$ is of the form
$$K=\frac{1}{2}K_{ij}\theta^i\wedge\theta^j,\quad {\rm where}\quad i,j=0,1,2,3,4,$$
and the above EDS is the same as 
$$d\omega=-\omega\wedge\omega+\frac{1}{2} K_{ij}\theta^i\wedge\theta^j.$$
%with the functions $(a_1,a_2,a_3,a_4,a_5)$, $(b_1,b_2,b_3,b_4)$, $(c_1,c_2,c_3)$, $(e_1,e_2)$, $f$, $(q_1,q_2),$ $(p_1,p_2,p_3),$ $(h_1,h_2,h_3,h_4)$ being the respective coefficients of the $\mathfrak{g}_2$-valued curvature functions $K_{ij}$.
%Assume in addition that the 1-forms $(\theta^0,\theta^1,\theta^2,\theta^3,\theta^4,\theta^5,\theta^6,$ $\Omega_1,\Omega_2,\Omega_3,\Omega_4,\Omega_5,\Omega_6,\Omega_7)$ satisfy the independence condition 
%\be
%\theta^0\dz\theta^1\theta^2\dz\theta^3\dz\theta^4\dz\theta^5\dz\theta^6\dz\Omega_1\dz\Omega_2\dz\Omega_3\dz\Omega_4\dz\Omega_5\dz\Omega_6\dz\Omega_7\neq 0,\label{np3}\ee
%on $\tilde{{\mathcal U}^5}$. 
%
%
%Then, there is a \emph{unique} solution $(\theta^0,\theta^1,\theta^2,\theta^3,\theta^4,\theta^5,\theta^6,$ $\Omega_1,\Omega_2,\Omega_3,\Omega_4,\Omega_5,$ $\Omega_6,\Omega_7)$, $(a_1,a_2,a_3,a_4,a_5)$, $(b_1,b_2,b_3,b_4)$, $(c_1,c_2,c_3)$, $(e_1,e_2)$, $f$, $h_1,h_2,h_3,h_4,$ $h_5,h_6,k_1,k_2,k_3$ to the system (\ref{np0})-(\ref{np3}) with the 1-forms $(\theta^0,\theta^1,\theta^2,\theta^3,\theta^4)$ as given in (\ref{su}) and (\ref{s3}). 
\end{theorem}

\subsection{From the EDS to  underlying structures}\label{edsd}
Suppose that the fourteen 1-forms $(\theta^0,\dots,\theta^6, \\ \Omega_1,\dots,\Omega_7)$ on $\mathcal{G}$ are linearly independent at each point, $\theta^0\dz\dots\theta^6\dz\Omega_1\dz\dots\dz\Omega_7\neq 0$, and  satisfy the EDS as in Theorem \ref{g2_gs}. 

\subsubsection{The underlying $(2,3,5)$-distribution and conformal metric} On the one hand,  we easily conclude the following:
\begin{itemize}
\item ${\mathcal G}$ is locally foliated by 9-dimensional submanifolds  tangent to the distribution $\mathcal{P}$ defined as the annihilator of the basis 1-forms $(\theta^0,\theta^1,\theta^2,\theta^3,\theta^4)$. That $\mathcal{P}$ is integrable follows immediately from the EDS, since it guarantees that
  $$\der\theta^k\dz\theta^0\dz\theta^1\dz\theta^2\dz\theta^3\dz\theta^4=0,\quad\quad\forall k=0,1,2,3,4.$$
\item The rank 2 distribution $\bar{\mathcal D}$ on ${\mathcal G}$ annihilated by the forms $(\theta^0,\theta^1,\theta^2,\theta^5\theta^6,\Omega_1,$ $\dots\Omega_7)$,
  $$\bar{\mathcal{D}}=\mathrm{ker}(\theta^0,\theta^1,\theta^2,\theta^5, \theta^6,\Omega_1,\dots,\Omega_7),$$ descends to a well defined rank 2-distribution ${\mathcal D}=\pi_*\bar{\mathcal D}$ on the  space $M={\mathcal G}/P$ of leaves of the distribution $\mathcal{P}$. 
  To see that this is the case,  consider the frame  $(X_0,\dots,X_6,Y_1,\dots,Y_7)$ dual to the coframe forms on ${\mathcal G}$. Then $\bar{\mathcal D}$ is spanned by $X_3$ and $X_4$, $$\bar{\mathcal D}=\Span(X_3,X_4).$$ To show that $\bar{\mathcal D}$  projects to a well-defined rank 2-distribution $M$ it is enough to show that, at each point of ${\mathcal G}$, the Lie derivatives of $X_3$ and $X_4$ with respect to the fiber directions $X_5,X_6,Y_1,Y_2,Y_3,Y_4,Y_5,Y_6,Y_7$ are spanned by no other vectors than the distribution  vectors, $X_3,X_4$, and the vertical vectors $X_5,X_6,Y_1,Y_2,Y_3,Y_4,Y_5,Y_6,Y_7$. Dually, this precisely means that in the considered EDS the terms $\theta^3\dz\theta^5$, $\theta^3\dz\theta^6$, $\Omega_i\dz\theta^3$, $\theta^4\dz\theta^5$, $\theta^4\dz\theta^6$, $\Omega_i\dz\theta^4$, $i=1,2,\dots,7$, cannot appear in the exterior derivatives of the forms $\theta^0$, $\theta^1$ and $\theta^2$. This is  the case for the EDS from Theorem \ref{g2_gs}.
  
   The  distribution ${\mathcal D}=\pi_*\bar{\mathcal D}$ on $M$ is $(2,3,5)$, since the EDS from Theorem \ref{g2_gs} guarantees the following expressions for the commutators $[X_3,X_4]=-X_0$, $[X_3,X_0]=X_1$ and $[X_4 ,X_0]=X_2$, where equality is considered \emph{modulo  terms vertical with respect to $\pi$}. 
  %Noting that after the projection to $M^5$ we have $\pi_*X_0\dz\pi_*X_1\dz\pi_*X_2\dz\pi_*X_3\dz\pi_*X_4\neq 0$ proves the claim in this item.  
  \item The conformal class of $(3,2)$ signature metrics $[g_{\mathcal D}]$ is represented by the bilinear form
    $$g_{\mathcal D}=\tfrac43(\theta^0)^2+2\theta^1\theta^4-2\theta^2\theta^3.$$
    The EDS from Theorem \ref{g2_gs} guarantees that the Lie derivatives of $g_{\mathcal D}$ with respect to its degenerate directions spanned by $X_5,X_6,Y_1,\dots,Y_7$ are always multiples of $g_{\mathcal D}$. Thus $g_{\mathcal D}$ descends to a well defined conformal class $[g_{\mathcal D}]$ of $(3,2)$ signature metrics on $M={\mathcal G}/{\mathcal P}$. 
\end{itemize}
\subsubsection{The corresponding Lie contact structure and  $(3,5,7)$-distribution}\label{underlying_Liecont}
On the other hand, the EDS in  from Theorem \ref{g2_gs}  can be viewed quite differently:
\begin{itemize}
\item Consider the rank 7-distribution
  ${\mathcal{Q}}$ on $\mathcal{G}$ defined as the annihilator of the
 seven linearly independent 1-forms $\theta^A$, $A=0,1,2,3,4,5,6$. This distribution is integrable  due to
  $$\der\theta^A\dz\theta^0\dz\theta^1\dz\theta^2\dz\theta^3\dz\theta^4\dz\theta^5\dz\theta^6=0,\quad \forall A=0,1,2,3,4,5,6.$$
   As such, it defines a foliation of ${\mathcal G}$ by 7-dimensional leaves, and a fibration
  $$Q\to{\mathcal G}\stackrel{\sigma}{\to}{\widetilde{M}}={\mathcal G}/Q,$$
  over the 7-dimensional leaf space $\widetilde{M}={\mathcal G}/Q$.
  \item The rank 6 distribution $\bar{\mathcal H}$ on ${\mathcal G}$ annihilated by the forms $(\theta^0,\Om_1,\dots,\Om_7)$,
  $$\bar{\mathcal H}=\mathrm{ker}(\theta^0,\Om_1,\dots,\Om_7),$$
    descends to a well defined rank 6 distribution ${\mathcal H}=\sigma_*\bar{\mathcal H}$ on the leaf  space $\widetilde{M}$. 
    
Moreover, using the EDS from Theorem \ref{g2_gs} and a similar reasoning as before, we easily show that the rank 6 distribution ${\mathcal H}=\sigma_*\bar{\mathcal H}$ is indeed a \emph{contact} distribution on $\widetilde{M}$. The one-form $\theta^0$ descends from ${\mathcal G}$ to a well-defined  \emph{line} of   contact forms $[\lambda]$ on $\widetilde{M}$.
\item Again using the EDS from Theorem \ref{g2_gs}, we show that the contact distribution $\mathcal H$ on $\widetilde{M}$ is equipped with additional structure. Consider the 2-form
\be
\rho=3\theta^1\dz\theta^6+3\theta^2\dz\theta^5+\theta^3\dz\theta^4,\label{rhok}\ee
and the symmetric rank 4 tensor
\be\begin{aligned}\label{Upsilonk}
  \Upsilon=2\theta^2(\theta^3){}^2&\theta^6-3(\theta^1){}^2(\theta^6){}^2-2\theta^1\theta^3\theta^4\theta^6-6\theta^1\theta^2\theta^5\theta^6+\\&2\theta^2\theta^3\theta^4\theta^5-2\theta^1(\theta^4){}^2\theta^5-3(\theta^2){}^2(\theta^5){}^2.\end{aligned}\ee 
Then the Lie derivatives of $\rho$ and $\Upsilon$ with respect to  the fiber directions $Y_A$  are
  $$\begin{aligned}{\mathcal L}_{Y_A}\rho=u_A \rho+
  \theta^0\dz\alpha_A \quad{\mathrm{and}}\quad
    {\mathcal L}_{Y_A}\Upsilon=v_A \Upsilon+\theta^0\odot\gamma_A,\end{aligned}$$
  where $u_A,v_A$ are functions, $\alpha_A$ are 1-forms,  and $\gamma_A$ are symmetric rank 3 tensors.     Since $\theta^0$ annihilates the distribution $\mathcal{H}$, $\rho$ and $\Upsilon$ descend to the respective objects $[\rho]$ and $[\Upsilon]$ on the distribution $\mathcal H$, where they are defined up to a scale, because some of the $u_A$, $v_A$ are non-zero. (In fact, the class of $\rho$ on $\mathcal H$ can be represented by $\der\theta^0$, so this is a line of symplectic forms on $\mathcal{H}$.)

\item The rank 3 distribution $\bar{\mathcal{S}}$ on the Cartan bundle $\mathcal{G}$ defined as
$$\bar{\mathcal S}=\mathrm{ker}(\theta^1,\theta^2,\theta^3,\theta^4,\Omega_1,\dots,\Omega_7)=\Span(X_0,X_5, X_6),$$
descends to a well-defined rank $3$-distribution $\mathcal{S}=\sigma_*\bar{\mathcal{S}}$ on $\widetilde{M}.$
This can be seen from the fact that 
 in the EDS from Theorem \ref{g2_gs} the exterior derivatives of the forms $\theta^1$, $\theta^2$, $\theta^3$ and $\theta^4$ do not contain terms of the form $\theta^0\wedge\Omega_i$, $\theta^5\wedge\Omega_i$ and $\theta^6\wedge\Omega_i$.
 
  One easily checks using the system that $[X_5,X_6]=0$, $[X_0,X_5]=-4X_3$, $[X_0,X_6]=4 X_4$, $[X_0,X_3]=-X_1$, $[X_0,X_4]=-X_2$ \emph{modulo vertical terms}. This shows that the first commutator $[\mathcal{S},\mathcal{S}]$ has rank $5$ (and is equal to  the lift $\widetilde{[\mathcal{D},\mathcal{D}]}$ of $[\mathcal{D},\mathcal{D}]$), and $[\mathcal{S},[\mathcal{S},\mathcal{S}]]=T\widetilde{M}.$ In particular,  the distribution $\mathcal{S}$ has growth vector $(3,5,7)$.
\end{itemize}

%\begin{remark}
Locally, the structure on $\widetilde{M}$ described above in terms of the contact distribution $\mathcal{H}$ equipped with the line of symmetric rank 4 tensors $[\Upsilon]$  is equivalent to a Lie contact structure as introduced in section \ref{liecontdef}. To see this, one needs to show that $[\Upsilon]$ reduces the structure group of the natural frame bundle of the contact distribution to the correct group $\widetilde{G}_0$. Now one verifies directly that the subalgebra of $\mathfrak{gl}(6)$  stabilizing $[\rho]$ and $[\Upsilon]$ is precisely   $\tilde{\g}_0=\mathfrak{gl}(2,\mathbb{R})\oplus\mathfrak{so}(1,2)$ in the proper representation.
%, and recall that $\mathfrak{sl}(2,\mathbb{R})\cong\mathfrak{so}(1,2)$.

\begin{remark}
Recall that $\mathfrak{so}(1,2)\cong\mathfrak{sl}(2,\mathbb{R})$. To see, algebraically,  where the tensor comes from, one can first verify  that there is precisely one trivial summand in the decomposition of the $\mathfrak{sl}(2,\mathbb{R})\oplus\mathfrak{sl}(2,\mathbb{R})$-representation $S^4(\mathbb{R}^2\boxtimes S^2\mathbb{R}^2)$ into irreducible components. Next, it is also not difficult to construct the invariant rank 4 tensor. Write an element $\psi\in\mathbb{R}^2\boxtimes S^2\mathbb{R}^2$ using index notation as ${\psi}^{A \dot{B}}$ and define a map
\begin{align*}
L(\psi):\mathbb{R}^2\to\mathbb{R}^2,\ \  {L(\psi)^{\dot{C}}}_{\dot{H}}=\psi^{A\dot{B}\dot{C}}\psi^{D\dot{E}\dot{F}}\epsilon_{AD}\epsilon_{\dot{B}\dot{E}}\epsilon_{\dot{F}\dot{H}},
\end{align*}
for volume forms  $\epsilon_{AB}\in\Lambda^2 \mathbb{R}^2$ and $\epsilon_{\dot{A}\dot{B}}\in\Lambda^2\mathbb{R}^2$. It turns out that the trace of this map is zero, but the trace of its square is not, and the unique up to constants invariant symmetric rank 4 tensor is 
\begin{align*}
\Upsilon(\psi)= \mathrm{Tr}(L(\psi)\circ L(\psi)).
\end{align*}
\end{remark}

\subsection{A particular solution  to the EDS in dimension $7$}\label{sformula}
Next we construct  the 1-forms $(\theta^0,\theta^1,\dots,\theta^6)$ explicitly with respect to a section  for a special class of distributions. In particular, this yields an explicit local description of the induced Lie contact structure (which however is not so simple even for the nice class of distributions that we consider).

Recall that we can specify a $(2,3,5)$-distribution ${\mathcal D}_F$ defined in a neighbourhood $\mathcal{U}^5$ around the origin of $\mathbb{R}^5$ with local coordinates $(x,y,p,q,z)$ by specifying a single function of five variables $F=F(x,y,p,q,z)$ such that $F_{qq}\neq 0$.  Let us consider a 
differentiable function $F=h(q)$ of one variable $q$ only. We assume that $h''\neq 0$. Then the distribution  $\mathcal D_{h}$ is given as the kernel  of the three one-forms
$$\omega^0=\der p-q\der x,\quad \om^1=\der y-p\der x, \quad \om^2=\der z-h\der x.$$
 The one-forms $(\omega^0,\omega^1,\omega^2)$ can be supplemented to a coframe $(\omega^i)$, $i=0,1,2,3,4$, on ${\mathcal U}^5$ given by:
\be
\begin{aligned}
\omega^0&=\der p-q\der x,\\
\om^1&=\der y-p\der x,\\
\om^2&=\der z-h\der x,\\
\om^3&=\der q,\\
\om^4&=\der x.
\end{aligned}
\label{s0}\ee
Now one introduces forms
\be
\bma\theta^0\\\theta^1\\\theta^2\\\theta^3\\\theta^4\ema=
\bma u_1&u_2&u_3&0&0\\
u_4&u_5&u_6&0&0\\
u_7&u_8&u_9&0&0\\
u_{10}&u_{11}&u_{12}&u_{13}&u_{14}\\
u_{15}&u_{16}&u_{17}&u_{18}&u_{19}
\ema\bma\om^0\\\om^1\\\om^2\\\om^3\\\om^4\ema,\label{s3}
\ee
with the 19 free parameters $(u_1,u_2,\dots, u_{19})$. It follows that there exists a choice of these parameters, in which the forms $(\theta^0,\theta^1,\dots,\theta^4)$ satisfy the EDS as in Theorem \ref{g2_gs}, with corresponding functions $(a_1, a_2, \dots, h_3,h_4)$ and 1-forms $(\theta^5,\theta^6,\Omega_1,\Omega_2,\dots,\Omega_7)$, such that
$$\theta^0\dz\theta^1\dz\theta^2\dz\theta^3\dz\theta^4\dz\theta^5\dz\theta^6\neq 0,$$
and
$$\Omega_i\dz \theta^0\dz\theta^1\dz\theta^2\dz\theta^3\dz\theta^4\dz\theta^5\dz\theta^6\equiv 0, \quad\forall i=1,2,\dots,7.$$
This means that there is an effective algorithm of solving the EDS of Theorem \ref{g2_gs} for forms $(\theta^0,\theta^1,\dots,\theta^6,\Omega_1,\Omega_2,\dots,\Omega_7)$ and the coefficients $(a_1,a_2,\dots,h_3,h_4)$ on a certain seven dimensional manifold, which we below parametrized by $(x,y,p,q,z,v,w)$. 

Explicitly, the forms  corresponding to this choice are given below:
\be{\small
\begin{aligned}\label{dupa}
  \theta^0&=\frac{v  {h''}^{4/3}}{9 w^4}\der y+\frac{ {h''}^{4/3}}{9 w^4}\der z-\frac{  \left(w+h'\right) {h''}^{4/3}}{9 w^4}\der p-\frac{ \left(v p-w q+h-q h'\right) {h''}^{4/3}}{9 w^4}\der x \\
  \theta^1&=-\frac{p  {h''}^{4/3}}{27 w^4} \der x+\frac{  {h''}^{4/3}}{27 w^4}\der y\\
  \theta^2&=\frac{v   {h''}^{5/3}}{27 w^5}\der y+\frac{  {h''}^{5/3}}{27 w^5}\der z-\frac{  h' {h''}^{5/3}}{27 w^5}\der p-\frac{  \left(v p+h-q h'\right) {h''}^{5/3}}{27 w^5}\der x  \\
     \theta^3&=\frac{v   h''}{3 w^3}\der y-\frac{  \big(-20 {h''}^4-4 w^2 {h^{(3)}}^2+3 w^2 h'' h^{(4)}\big)}{90 w^3 {h''}^3}\der z-\\&\frac{  \big(40 w {h''}^4+20 h' {h''}^4+10 w^2 {h''}^2 h^{(3)}+4 w^2 h' {h^{(3)}}^2-3 w^2 h' h'' h^{(4)}\big)}{90 w^3 {h''}^3}\der p+\\&\frac{1}{90 w^3 {h''}^3}  \big(30 w^2 {h''}^3-30 v p {h''}^4+40 w q {h''}^4-20 h {h''}^4+20 q h' {h''}^4+10 w^2 q {h''}^2 h^{(3)}-\\&4 w^2 h {h^{(3)}}^2+4 w^2 q h' {h^{(3)}}^2+3 w^2 h h'' h^{(4)}-3 w^2 q h' h'' h^{(4)}\big)\der x\\
       \theta^4&=-\frac{{h''}^{4/3}}{3 w^2}\der q+\frac{v^2   {h''}^{4/3}}{9 w^4}\der y-\frac{v   \big(-10 {h''}^4-4 w^2 {h^{(3)}}^2+3 w^2 h'' h^{(4)}\big)}{90 w^4 {h''}^{8/3}}\der z+\\&\frac{v  \big(-10 h' {h''}^4-10 w^2 {h''}^2 h^{(3)}-4 w^2 h' {h^{(3)}}^2+3 w^2 h' h'' h^{(4)}\big)}{90 w^4 {h''}^{8/3}}\der p-\\&\frac{v}{90 w^4 {h''}^{8/3}}   \big(-30 w^2 {h''}^3+10 v p {h''}^4+10 h {h''}^4-10 q h' {h''}^4-10 w^2 q {h''}^2 h^{(3)} +\\&4 w^2 h {h^{(3)}}^2-4 w^2 q h' {h^{(3)}}^2-3 w^2 h h'' h^{(4)}+3 w^2 q h' h'' h^{(4)}\big)\der x\\
  \theta^5&=\frac{\der w }{{h''}^{1/3}}+\frac{\big(10 {h''}^4-10 w {h''}^2 h^{(3)}+4 w^2 {h^{(3)}}^2-3 w^2 h'' h^{(4)}\big)}{30 {h''}^{10/3}}\der q+\\&\frac{v   \big(-5 {h''}^6+40 w^3 {h^{(3)}}^3-45 w^3 h'' h^{(3)} h^{(4)}+9 w^3 {h''}^2 h^{(5)}\big)}{90 w^2 {h''}^{16/3}}\der z+\\&\frac{v}{90 w^2 {h''}^{16/3}}   \big(-15 w {h''}^6+5 h' {h''}^6-12 w^3 {h''}^2 {h^{(3)}}^2-40 w^3 h' {h^{(3)}}^3+9 w^3 {h''}^3 h^{(4)}+\\&45 w^3 h' h'' h^{(3)} h^{(4)}-9 w^3 h' {h''}^2 h^{(5)}\big)\der p-\frac{v}{90 w^2 {h''}^{16/3}}   \big(-15 w q {h''}^6-5 h {h''}^6+5 q h' {h''}^6-\\&12 w^3 q {h''}^2 {h^{(3)}}^2+40 w^3 h {h^{(3)}}^3-40 w^3 q h' {h^{(3)}}^3+9 w^3 q {h''}^3 h^{(4)}-45 w^3 h h'' h^{(3)} h^{(4)}+\\&45 w^3 q h' h'' h^{(3)} h^{(4)}+9 w^3 h {h''}^2 h^{(5)}-9 w^3 q h' {h''}^2 h^{(5)}\big)\der x\\
         % \end{aligned} \ee
 % \be{\small  \begin{aligned}
  \theta^6&=-\der v +\frac{v^2  \big(-4 {h^{(3)}}^2+3 h'' h^{(4)}\big)}{90 w {h''}^3}\der z +\frac{v \big(10 {h''}^4-4 w^2 {h^{(3)}}^2+3 w^2 h'' h^{(4)}\big)}{30 w {h''}^3}\der q+\\&\frac{v^3   \big(40 {h^{(3)}}^3-45 h'' h^{(3)} h^{(4)}+9 {h''}^2 h^{(5)}\big)}{90 {h''}^5}\der y-\frac{v^2}{90 w^2 {h''}^5}   \big(5 {h''}^6-10 w {h''}^4 h^{(3)}-\\&12 w^2 {h''}^2 {h^{(3)}}^2-4 w h' {h''}^2 {h^{(3)}}^2+40 w^3 {h^{(3)}}^3+9 w^2 {h''}^3 h^{(4)}+3 w h' {h''}^3 h^{(4)}-45 w^3 h'' h^{(3)} h^{(4)}+\\&9 w^3 {h''}^2 h^{(5)}\big)\der p-\frac{v^2}{90 w^2 {h''}^5}   \big(30 w {h''}^5-5 q {h''}^6+10 w q {h''}^4 h^{(3)}+12 w^2 q {h''}^2 {h^{(3)}}^2-\\&4 w h {h''}^2 {h^{(3)}}^2+4 w q h' {h''}^2 {h^{(3)}}^2+40 v w^2 p {h^{(3)}}^3-40 w^3 q {h^{(3)}}^3-9 w^2 q {h''}^3 h^{(4)}+3 w h {h''}^3 h^{(4)}-\\&3 w q h' {h''}^3 h^{(4)}-45 v w^2 p h'' h^{(3)} h^{(4)}+45 w^3 q h'' h^{(3)} h^{(4)}+9 v w^2 p {h''}^2 h^{(5)}-9 w^3 q {h''}^2 h^{(5)}\big)\der x.
\end{aligned}  } \ee
We could also write down the remaining forms $(\Omega_1,\Omega_2,\dots,\Omega_7)$ that together with the above $(\theta^0,\theta^1,\dots,\theta^6)$ satisfy the EDS from Theorem \ref{g2_gs}, but since they are not interesting for the rest of our paper we will skip them. 

\subsection{Examples}
 The particular solution $(\theta^0,\theta^1,\dots,\theta^6)$ constructed above, enables us to write down  the structural tensors associated with the $(2,3,5)$ distribution
  \begin{align}\label{distribution_example}
  {\mathcal D}_h~=~\Span(~\partial_x+p\partial_y+q\partial_p+h(q)\partial_z,~\partial_q~),\end{align}
    explicitly in the coordinates $(x,y,p,q,z;v,w)$. 
  %   In particular, one obtains an explicit formula for the conformal symmetric  fourth rank tensor  $\Upsilon$  on the contact distribution $\mathcal{H}=\mathrm{ker}(\theta^0)$. 
   
    It should be clear that the coordinates $(x,y,p,q,z)$ parametrize the 5-manifold $M$ on which the distribution ${\mathcal D}_h$ resides, and that  $(v,w)$ are the fiber coordinates of the bundle $\widetilde{M}\to M$. In particular, $(v,w)$ locally parameterize  directions $\ell(v,w)=\mathrm{dir}(\xi(v,w))$ in the 3-distribution $[{\mathcal D}_h,{\mathcal D}_h]$ as follows:
    $$\xi(v,w)=\partial_x+p\partial_y+q\partial_p+h\partial_z+\frac{v}{h''}\partial_q+\frac{w}{h''}(\partial_p+h'\partial_z).$$
Note that in this parametrization the directions \emph{transverse} to the 2-distribution ${\mathcal D}_h$ have $w\neq 0$, and that $w\equiv 0$ corresponds to the directions in the 2-distribution ${\mathcal D}_h$. Thus, when the coordinate $w\to 0$ we approach points $(x,y,p,q,z,v)$ of the 6-dimensional boundary $\mathbb{P}({\mathcal D}_h)$ of $\widetilde{M}\cong\mathbb{P}([{\mathcal D}_h,{\mathcal D}_h])\setminus\mathbb{P}({\mathcal D}_h)$.

In the following we will restrict our examples to the distributions $\mathcal{D}_{h}$ with
\be\label{h_example}
h(q)=\frac{1}{k (k-1)}q^k, \quad \mathrm{where} \quad k\in\mathbb{R}, \quad k\neq 0, 1.
\ee
Since in such case ${\mathcal D}_h$ is totally determined by a real number $k$, we will denote these distributions by ${\mathcal D}_k$. We have excluded the cases $k= 0, 1$ because they do not correspond to $(2,3,5)$ distributions. 
\subsubsection{Conformal metric on $M$} 
For the class of examples given by \eqref{h_example} the conformal class of metrics $[g_{\mathcal{D}_{k}}]$ may be represented by
%%/home/pawel/worksheets/distribution_C4/totally_new_examples_q_to_k_paper_metric.mw
    \be\begin{aligned}\label{metric_example}
      g_{{\mathcal D}_k}&=(k-1)^2(9k^2-9k+2)q^2\der x^2-2k(k-1)(9k^2-9k-8)q\der x \der p+\\&30k^2(k-1)^2p \der x \der q-4k(k-1)^2(3k^2+2k-1)q^{2-k}\der x \der z-\\&30k^2(k-1)^2\der y\der q+k^2(9k^2-9k+2)\der p^2+\\&4k^2(k-1)(3k^2-8k+4)q^{1-k}\der p\der z-k^2(k-1)^2(k^2-k-2)q^{2-2k}\der z^2.\end{aligned}\ee
It is well known \cite{nurowski-metric} that this metric is conformally flat if and only if the corresponding distribution $\mathcal{D}_{k}$ is flat, and this happens \cite{cartan-cinq} precisely in the four cases when $k\in\{2,\tfrac{2}{3},\tfrac{1}{3},-1\}$. For example for $k=2$ we get the conformally flat metric 
    \be
      g_{{\mathcal D}_2}=4\Big(30  \der x \der z-5q^2 \der x^2-20  \der p^2+10 q \der p \der x-30 p  \der q \der x+30  \der q \der y\Big).\label{conffl}\ee

Now if $k\notin\{2,\tfrac{2}{3},\tfrac{1}{3},-1\}$ the distribution $\mathcal{D}_{k}$ has  $7$-dimensional symmetry algebra (the submaximal dimension) spanned by

\begin{align*}
&X_1=\partial_x,\quad X_2=\partial_y,\quad X_3=\partial_z,\quad X_4=\partial_p+x \partial_y,\\
&X_5=x\partial_x-p\partial_p-2 q \partial_q+(1-2k)z\partial_z,\quad X_6=y\partial_y+p\partial_p+q\partial_q+k z \partial_z,\\
&X_7=q^{k-1}\partial_x+\big(p q^{k-1}+(1-k)z\big)\partial_y+\frac{k-1}{k}q^k\partial_p+\frac{q^{2k-1}}{k(2k-1)}\partial_z.
\end{align*}
The conformal class represented by \eqref{metric_example}  has $9$-dimensional symmetry algebra,   spanned by $X_1,\dots, X_7$ and the two \emph{additional} generators

\begin{align*}
&X_8=q^{-\frac{1}{2}+\frac{\sqrt{10k^2-10k+5}}{10}}\big(\partial_x+p\partial_y+\tfrac{3k^2-2\sqrt{10k^2-10k+5}-3k+4}{(3k-2)(k-2)}q\partial_p+2\tfrac{4k^2-4k+2-k\sqrt{10k^2-10k+5}}{(3k-2)(k-2)k(k-1)}q^{-k}\partial_z\big),\\
&X_9=q^{-\frac{1}{2}+\frac{\sqrt{10k^2-10k+5}}{10}}\big(\partial_x+p\partial_y+\tfrac{3k^2+2\sqrt{10k^2-10k+5}-3k+4}{(3k-2)(k-2)}q\partial_p+2\tfrac{4k^2-4k+2+k\sqrt{10k^2-10k+5}}{(3k-2)(k-2)k(k-1)}q^{k}\partial_z\big).\\
\end{align*}

It is instructive to look at the symmetries in one of the flat cases, say $k=2$. One sees that in this case $X_8$ and $X_9$ are singular, but the rescaling by a factor $(k-2)$ regularizes them at $k=2$. These however, in the limit $k\to 2$, lead to one  symmetry only, namely to $Z_1=\mathrm{lim}_{k\to 2}X_9=\partial_p+q\partial_z$, since the limit of the regularized $X_8$ is zero. In this case the 8 conformal symmetries $(X_1,X_2,\dots,X_7,X_9{}')$ are of course extendible to the full 21-dimensional algebra of symmetries $\soa(3,4)$.

      We close this section providing the full algebra of symmetries of the distribution ${\mathcal D}_k$ with $k=2$ and the full algebra of conformal symmetries of $[g_{{\mathcal D}_k}]$ in such case. On top of the 7 symmetries $(X_1,X_2,\dots, X_7)$ with $k=2$ this distribution has additional 7 symmetries, so that its full algebra of symmetries has dimension 14. The remaining 7 symmetries look as follows:
      \be\begin{aligned}
        Y_1=&\tfrac12 x^2\partial_y+x\partial_p+\partial_q+p\partial_z,\\
        Y_2=&\tfrac16 x^3\partial_y+\tfrac12x^2\partial_p+x\partial_q+(xp-y)\partial_z,\\
        Y_3=& x^2\partial_x+3xy\partial_y+(3y+xp)\partial_p+(4p-qx)\partial_q+2p^2\partial_z,\\
        Y_4=& (8p-6qx)\partial_x+(4p^2+6xz-6pqx)\partial_y+(6z-3q^2x)\partial_p-2q^2\partial_q-q^3x\partial_z,\\
        Y_5=& (16xp-12y-6qx^2)\partial_x+(6x^2z+8p^2x-6pqx^2)\partial_y+(12xz+4p^2-3q^2x^2)\partial_p+\\&(12z+4pq-4q^2x)\partial_q+(12pz-q^3x^2)\partial_z,\\
        Y_6=& (24px^2-6qx^3-36xy)\partial_x+(12p^2x^2+6x^3z-36y^2-6pqx^3)\partial_y+\\&(12p^2x+18x^2z-3q^2x^3-36py)\partial_p+(12pqx-6q^2x^2-24p^2+36xz)\partial_q+\\&(36pxz-8p^3-q^3x^3-36yz)\partial_z,\\
        Y_7=& (12p^2-18qy)\partial_x+(8p^3-18pqy+18yz)\partial_y+(18pz-9q^2y)\partial_p+(18qz-6pq^2)\partial_q+\\&(18z^2-3q^3y)\partial_z.
        \end{aligned}\label{next7}\ee
The 14-dimensional Lie algebra spanned by $(X_1,X_2,\dots,X_7,Y_1,Y_2,\dots,Y_7)$ is isomorphic to the split real form of the exceptional simple Lie algebra $\mathfrak{g}_2$. As for the conformal symmetries of $[g_{{\mathcal D}_2}]$: we have the 14 conformal symmetries of the distribution, $(X_1,X_2,\dots,X_7,Y_1,Y_2,\dots,Y_7)$, forming the Lie algebra of $\mathfrak{g}_2$, but also 7 additional conformal symmetries given by:
$$\begin{aligned}
  Z_1=&\partial_p+q\partial_z,\\
  Z_2=&\partial_x+p\partial_y+\tfrac34q\partial_p+\tfrac14q^2\partial_z,\\
  Z_3=&\partial_q+\tfrac14x\partial_p+\tfrac14qx\partial_z,\\
  Z_4=&4px\partial_y+(3qx+6p)\partial_p+12q\partial_q+(q^2x+2qp+12z)\partial_z,\\
  Z_5=&4x^2\partial_x+4px^2\partial_y+(3qx^2+4px-6y)\partial_p+8(qx-p)\partial_q+(q^2x^2+4xqp-6qy)\partial_z,\\
  Z_6=&12qx\partial_x+(12xqp+8p^2-12xz)\partial_y+(6q^2x+12qp)\partial_p+12q^2\partial_q+(2q^3x+4q^2p+12qz)\partial_z,\\
  Z_7=&4(px-3y)\partial_x+4p(px-3y)\partial_y+(3xqp-2p^2-9qy+3xz)\partial_p+(2q^2x-8qp+12z)\partial_q+\\&(xpq^2-2p^2q-3yq^2+3zqx)\partial_z.
\end{aligned}
  $$
  The 21-dimensional algebra generated by $(X_1,X_2,\dots,X_7,Y_1,Y_2,\dots,Y_7,Z_1,\dots,Z_7)$ is isomorphic to the Lie algebra $\soa(3,4).$ 
  %and is the full algebra of conformal symmetries of the conformal structure $[g_{{\mathcal D}_2}]$. 

\subsubsection{The Lie contact structure}
The Lie contact structure $([\lambda],[\Upsilon])$ on $\widetilde{M}$ associated with the distribution $\mathcal{D}_h$ is totally expressible in terms of the forms $(\theta^0,\theta^1,\dots,\theta^6)$ as written in Section \ref{sformula}, formulas (\ref{rhok}), (\ref{Upsilonk}). For $h(q)=\frac{q^k}{k(k-1)}$, the line of contact forms can be represented by
\be
  \lambda=\der z- ( w+\frac{q^{k-1}}{k-1})\der p +v \der y + (w q -v p +\frac{q^k}{k})\der x.\label{cfk}
\ee
To get this, we took $\theta^0$ from (\ref{dupa}), calculated it for $h=\frac{q^k}{k(k-1)}$ and rescaled, so that the term at $\der z$ is equal to one.
One easily checks that
\begin{align*}
\der \lambda\wedge\der \lambda\wedge\der\lambda\wedge\lambda=-6 w \der x\wedge\der y\wedge\der p\wedge\der q\wedge\der z \wedge \der v \wedge\der w,
\end{align*}
so $\lambda$ is a contact form everywhere on $\widetilde{M}$ except the boundary $w=0$. Even in the simple case we are considering, the structural tensor $\Upsilon$ on ${\mathcal D}_k$, when written via the formula \eqref{Upsilonk} in coordinates $(x,y,p,q,z,v,w)$, has a very ugly look. For this reason we will not write it here. Instead we determine the symmetries of the Lie contact structure $([\lambda],[\Upsilon])$ on $\widetilde{M}=\mathbb{P}([{\mathcal D}_k,{\mathcal D}_k])\setminus\mathbb{P}({\mathcal D}_k)$ with this ugly $\Upsilon$.

In general, an infinitesimal symmetry of a Lie contact structure $([\lambda],[\Upsilon])$ on $\widetilde{M}$ is a vector field $X$ on $\widetilde{M}$ such that
\be
(\mathcal{L}_X\lambda)\wedge\lambda=0,\quad \mathrm{and}\quad\mathcal{L}_X\Upsilon=f \Upsilon +\lambda\odot\tau,\label{symeq}
\ee
 where $\tau$ is a rank $3$ tensor and $f$ is a function on $\widetilde{M}$. We calculated the infinitesimal symmetries of the Lie contact structure $([\lambda],[\Upsilon])$ with $\lambda$ as in (\ref{cfk}) and $\Upsilon$ determined by (\ref{Upsilonk}), (\ref{dupa}) with $h=\frac{q^k}{k(k-1)}$, obtaining the following proposition.
%%/home/pawel/worksheets/distribution_C4/totally_new_examples_q_to_k_paper_metric_solving_1.mw
      \begin{proposition}\label{prop_symmetries}
        If $k\notin\{2,\tfrac23,\tfrac13,0,1,-1\}$ the algebra of infinitesimal symmetries of the Lie contact structure $([\lambda],[\Upsilon])$ on $\widetilde{M}=\mathbb{P}([{\mathcal D}_k,{\mathcal D}_k])\setminus\mathbb{P}({\mathcal D}_k)$ is seven dimensional and is spanned by the infinitesimal symmetries:
        \begin{align*}\begin{aligned}
          \tilde{X}_1=&\partial_x,\quad\tilde{X}_2=\partial_y,\quad\tilde{X}_3=\partial_z,\quad\tilde{X}_4=\partial_p+x\partial_y,\\
          \tilde{X}_5=&x\partial_x-p\partial_p-2q\partial_q+(1-2k)z\partial_z+(1-2k)v\partial_v+2(1-k)w\partial_w,\\
          \tilde{X}_6=&y\partial_y+p\partial_p+q\partial_q+kz\partial_z+(k-1)v\partial_v+(k-1)w\partial_w,\\
          \tilde{X}_7=&q^{k-1}\partial_x+\big(pq^{k-1}+(1-k)z\big)\partial_y+\frac{k-1}{k}q^k\partial_p+\frac{q^{2k-1}}{k(2k-1)}\partial_z+(1-k)v^2\partial_v+(1-k)vw\partial_w.
                  \end{aligned}\end{align*}
      \end{proposition}
      
       \begin{remark}
        Note that the seven symmetries $(\tilde{X}_1,\tilde{X}_2,\dots,\tilde{X}_7)$ above correspond to the seven symmetries $(X_1,X_2,\dots,X_7)$ of the distribution ${{\mathcal D}_k}$ defining the Lie contact structure $([\lambda],[\Upsilon])$. Explicitly note that we have: $\tilde{X}_i=X_i+a_i\partial_v+b_i\partial_w$, $i=1,2,\dots,7$, with specific functional cooeficients $a_i$ and $b_i$. We remark that we obtained $(\tilde{X}_1,\tilde{X}_2,\dots,\tilde{X}_7)$ by directly solving the symmetry equations (\ref{symeq}), and \emph{not} by assuming that the symmetry $\tilde{X}_i$ have the form $\tilde{X}_i=X_i+a_i\partial_v+b_i\partial_w$.

        There is, however, more direct way of getting these seven symmetries. This is related to the general fact that every symmetry of a $(2,3,5)$ distribution $\mathcal D$ induces a symmetry of its twistor Lie contact structure. The simplest way of seeing this is via the \emph{prolongation lift} (or simply prolongation) of an infinitesimal symmetry $X$ of a distribution. We claim that given ${\mathcal D}$ and its infinitesimal symmetry $X$ on $M$ we can lift it to a vector field $\tilde{X}$ on the twistor bundle $\widetilde{M}$. This is done point by point as follows:

        Suppose that we want to lift $X_p$, i.e. the vector defined by an infinitesimal symmetry $X$ at $p\in M$, from point $p$ to a point $(p,\ell)$ in the fiber in $\widetilde{M}$ over $p$. At $p$ the point $(p,\ell)$ defines a direction $\ell$ in the 3-distribution $[{\mathcal D},{\mathcal D}]$. We transport this direction by a flow $\phi(t)$ of $X$ along its integral curve $p(t)$ passing through $p$, $p(0)=p$. This defines a direction $\ell(t)=\phi_*(t)\ell$ at every point of the curve $p(t)$. Thus starting with $\ell(0)=\ell$ at $p(0)=p$, we have a direction $\ell(t)$ at $p(t)$ for every $t$. Since $X$ is a symmetry of a $(2,3,5)$ distribution, its flow preserves the 3-distribution, so for any value of $t$ the direction $\ell(t)$ sits in the 3-distribution. Thus, choosing a point $\ell$ at a fiber of $p$, at each point $p(t)$ of an integral curve of a symmetry vector field $X$ we have a direction $\ell(t)$ in the 3-distribution. We thus have a curve $(p(t),\ell(t))$ in the bundle $\widetilde{M}$, which starts at $(p,\ell)$ and which projects to $p(t)$. The tangent vector $\tilde{X}_{(p,\ell)}$ to this curve at $t=0$, is the \emph{prolongation lift} of the symmetry vector $X_p$ from $p\in M$ to $(p,\ell)\in \widetilde{M}$. By repeating this procedure for all pairs $(p,\ell)\in\widetilde{M}$ we define a vector field $\tilde{X}$ on $\widetilde{M}$ consisting of vectors $\tilde{X}_{(p,\ell)}$. We call $\tilde{X}$ the prolongation of $X$. It follows from the construction that the prolongation $\tilde{X}$ of an infinitesimal symmetry $X$ of a $(2,3,5)$ distribution $\mathcal D$ is an infinitesimal symmetry of the corresponding Lie contact structure $([\lambda],[\Upsilon])$ on $\widetilde{M}$.

        Finishing the remark we stress that \emph{all} infinitesimal symmetries of the Lie contact structures $([\lambda],[\Upsilon])$ on $\widetilde{M}=\mathbb{P}([{\mathcal D}_k,{\mathcal D}_k])\setminus\mathbb{P}({\mathcal D}_k)$ with all $k\notin\{2,\tfrac23,\tfrac13,0,1,-1\}$ are just prolongations of infinitesimal symmetries of the distribution ${\mathcal D}_k$. We have proven this by explicitly solving the symmetry equations and finding all their solutions.
        %{\color{red} Shall I give explicit example of a symmetry $X$ of $\mathcal{D}$ in which I explicitely calculate the coefficients $(a,b)$?}
        
        \end{remark}

\subsubsection{A $(3,5,7)$ distribution}
It is also interesting to look at the infinitesimal symmetries of the prolongation 
$$\mathcal{S}= \mathrm{Span} \Big(\, \partial_x +p\,\partial_y+q\, \partial_p+ h \, \partial_z+ \frac{v}{h''}\, \partial_q+ \frac{w}{h''}(\partial_p+q\,\partial_z),\, \partial_v,\, \partial_w \,\Big)$$  of $\mathcal{D}_h$.
For \eqref{h_example} and  $k\notin\{2,\tfrac{2}{3},\tfrac{1}{3},-1\}$, the $7$ lifts of infinitesimal symmetries of the distribution $\mathcal{D}_k$ from Proposition \ref{prop_symmetries}
clearly preserve the $(3,5,7)$ distribution $\mathcal{S}$. We calculated that  all infinitesimal symmetries of ${\mathcal S}={\mathcal S}_k$ are contained in the span of these $7$ symmetries. We further calculated that in the flat case  $k=2$, the symmetry algebra of the distribution $\mathcal{S}_{2}$ is precisely $\mathfrak{g}_2$  (the symmetry algebra of the Lie contact structure is of course $\mathfrak{so}(3,4)$ in this case; see the end of this chapter for details).
It turns out that for any $(2,3,5)$ distribution, the prolongation $\mathcal{S}$ has the same symmetry algebra as the underlying $(2,3,5)$-distribution $\mathcal{D}$:

\begin{proposition}
For any $(2,3,5)$ distribution $\mathcal{D}\subset T M$, the infinitesimal symmetries of the prolongation $\mathcal{S}\subset T\widetilde{M}$ are precisely the lifts of the infinitesimal symmetries of $\mathcal{D}$.
\end{proposition}
\begin{proof}  
Since the construction is natural, every infinitesimal  symmetry $X\in\mathfrak{X}(M)$ of the $(2,3,5)$-distribution $\mathcal{D}$ lifts to a vector field $\widetilde{X}\in\mathfrak{X}(\widetilde{M})$ that preserves the induced geometric structure on the twistor bundle. In particular, it defines a symmetry of the $(3,5,7)$ distribution $\mathcal{S}$.
It remains to show that every infinitesimal  symmetry of $\mathcal{S}$ projects to an infinitesimal symmetry of $\mathcal{D}$.
Consider the tensorial map $\Lambda^2\mathcal{S}\to [\mathcal{S},\mathcal{S}]/\mathcal{S}$ induced by the Lie bracket. At every point  this is a surjective map from a  $3$-dimensional to a $2$-dimensional space and thus it has a $1$-dimensional kernel spanned by a decomposable element. 
%The corresponding distribution is the unique rank $2$ distribution $\mathcal{D}'$ such that $[\mathcal{D}',\mathcal{D}']=\tilde{\mathcal{S}}$ and hence it is precisely the vertical bundle $\tilde{\mathcal{V}$.
So this defines a rank $2$ distribution on $\widetilde{M}$. Since the vertical distribution $\mathcal{V}$ for $\widetilde{M}\to M$ is evidently contained in this rank $2$-distribution  and of the same dimension, the two coincide. Note that this means that the vertical bundle is characterised as the unique rank $2$ subbundle in $\mathcal{S}$ such that Lie brackets of its sections are again contained in $\mathcal{S}$. This  in particular implies that any infinitesimal  symmetry $\widetilde{X}\in\mathfrak{X}(\widetilde{M})$ of $\mathcal{S}$ also preserves the vertical bundle  $\mathcal{V}$ and thus it is projectable to a vector field $X\in\mathfrak{X}(M)$. Moreover, $\widetilde{\xi}$ also preserves $ [\mathcal{S},\mathcal{S}]$, and since 
$ [\mathcal{S},\mathcal{S}]=\widetilde{[\mathcal{D},\mathcal{D}]}$, then naturality of the Lie bracket implies that $X$ preserves $[\mathcal{D},\mathcal{D}]$. By the same line of argument as above, $\mathcal{D}$ can be characterised as the unique rank $2$ subbundle in  $[\mathcal{D},\mathcal{D}]$ such that Lie brackets of its sections are again contained in  $[\mathcal{D},\mathcal{D}]$, and this implies that $\xi$ is an infinitesimal symmetry for the $(2,3,5)$-distribution $\mathcal{D}$.
\end{proof}

\subsubsection{Flat Lie contact structure} We conclude this chapter with the flat case, corresponding to the $(2,3,5)$ distribution with $h(q)=\tfrac12 q^2$, i.e. $k=2$. In this case we have
$$\lambda=\der z- ( w+q)\der p +v \der y + (w q -v p +\frac{q^2}{2})\der x$$
and the conformal tensor $\Upsilon$ on $\mathrm{ker}(\lambda)$ can be represented by
\begin{align*}
\Upsilon=& -3v^2(2vp-2wq-q^2)\der x^4+ 6 v^3 \der x^3 \der y- 6 v^2(w+q)\der p\der x^3 +\\
&6 v  (3vp-2wq-q^2) \der q\der x^3 -2(-9 w^2q-12wq^2+9vpw+9pvq-4q^3)\der v\der x^3 +\\
&6 v (-q^2+3vp-3wq)\der w\der x^3+ 3 v^2\der p^2 \der x^2 -3(-2wq-q^2+6vp)\der q^2\der x^2-\\
&9p^2\der v^2\der x^2 - 324 q^2  \der w^2 \der x^2 + 9 v  (w+q) \der p\der q \der x^2+\\
&6  (-3w^2 -8wq -4 q^2 +3vp) \der p \der v \der x^2 +6 v  (3 w +2 q) \der p \der w \der x^2+\\
&18 p  (w+q)\der q\der v \der x^2  - 6  (-3 w q+6 v p-q^2)\der q\der w\der x^2 - 18 v^2 \der q\der y\der x^2+ \\
&18  pq \der v\der w\der x^2+
18 v  (w+q) \der v \der y \der x^2- 18 v^2 \der w \der y \der x^2- 6 v \der p^2 \der q \der x+\\
&24  (w+q) \der p^2\der v \der x    -6v  \der p^2 \der w \der x   -6  (w+q)  \der p \der q ^2 \der x -\\
& 18  p \der p\der q \der v \der x- 6  (3 w +2 q) \der p \der q \der w \der x -18  p \der p \der v \der w \der x -\\
&18 v  \der p \der v \der y \der x  + 18 q  \der p \der w^2 \der x + 6 p  \der q^3 \der x + 18 p  \der q^2 \der w \der x+\\
&  18 v \der q^2 \der x \der y - 18  (w+q) \der q\der v \der y \der x+ 36 v  \der q \der w \der x \der y  +
 18 p  \der v^2 \der x \der y - \\&  18 q  \der v \der w \der x \der y- 8  \der p^3\der v + 3  \der p^2 \der q^2  + 
6  \der p^2 \der q \der w - 9  \der p^2 \der w^2+ \\& 18  \der p \der q \der v \der y + 18  \der p \der v \der w \der y -
 6  \der q^3 \der y -18  \der q^2 \der w \der y -9  \der v^2 \der y^2.\\
\end{align*}
The infinitesimal symmetries of the Lie contact structure $([\lambda],[\Upsilon])$ form a Lie algebra $\soa(3,4)$ and are naturally grouped as  $(\tilde{X}_1,\dots,\tilde{X}_7)$,  $(\tilde{Y}_1,\dots,\tilde{Y}_7)$ and  $(\hat{Z}_1,\dots,\hat{Z}_7)$, where we have: 
\begin{itemize}
\item 
The first 7 symmetries are just prolongations $(\tilde{X}_1,\dots,\tilde{X}_7)$ of the 7 symmetries $(X_1,\dots,X_7)$ of the distribution ${\mathcal D}_k$, as given in Proposition \ref{prop_symmetries}, and restricted to the case $k=2$: 
 \begin{align*}\begin{aligned}
          \tilde{X}_1=&\partial_x,\quad\tilde{X}_2=\partial_y,\quad\tilde{X}_3=\partial_z,\quad\tilde{X}_4=\partial_p+x\partial_y,\\
          \tilde{X}_5=&x\partial_x-p\partial_p-2q\partial_q-3z\partial_z-3v\partial_v-2w\partial_w,\\
          \tilde{X}_6=&y\partial_y+p\partial_p+q\partial_q+2z\partial_z+v\partial_v+w\partial_w,\\
          \tilde{X}_7=&q\partial_x+\big(pq-z\big)\partial_y+\tfrac12q^2\partial_p+\tfrac16q^3\partial_z-v^2\partial_v-vw\partial_w.
 \end{aligned}\end{align*}
\end{itemize}
\begin{itemize}
\item 
The second group of symmetries are the lifts of the 7 symmetries $(Y_1,\dots,Y_7)$ of the flat distribution ${\mathcal D}_2$ given in (\ref{next7}).
  \begin{align*}\begin{aligned}
        \tilde{Y}_1=&\tfrac12 x^2\partial_y+x\partial_p+\partial_q+p\partial_z,\\
        \tilde{Y}_2=&\tfrac16 x^3\partial_y+\tfrac12x^2\partial_p+x\partial_q+(xp-y)\partial_z+\partial_v,\\
        \tilde{Y}_3=& x^2\partial_x+3xy\partial_y+(3y+xp)\partial_p+(4p-qx)\partial_q+2p^2\partial_z-(3vx-3w-3q)\partial_v-wx\partial_w,\\
        \tilde{Y}_4=& (8p-6qx)\partial_x+(4p^2+6xz-6pqx)\partial_y+(6z-3q^2x)\partial_p-2q^2\partial_q-q^3x\partial_z+\\
        &(6v^2x-6vw-6vq)\partial_v+(6vwx-6w^2-4wq)\partial_w,\\
        \tilde{Y}_5=& (16xp-12y-6qx^2)\partial_x+(6x^2z+8p^2x-6pqx^2)\partial_y+(12xz+4p^2-3q^2x^2)\partial_p+\\&(12z+4pq-4q^2x)\partial_q+(12pz-q^3x^2)\partial_z+\\&(6v^2x^2-12vwx-12vqx+12wq+6q^2)\partial_v+
        (6vwx^2-12w^2x-8wqx+4wp)\partial_w,\\
        \tilde{Y}_6=& (24px^2-6qx^3-36xy)\partial_x+(12p^2x^2+6x^3z-36y^2-6pqx^3)\partial_y+\\&(12p^2x+18x^2z-3q^2x^3-36py)\partial_p+(12pqx-6q^2x^2-24p^2+36xz)\partial_q+\\&(36pxz-8p^3-q^3x^3-36yz)\partial_z+\\&(6v^2x^3-18vw x^2-18vqx^2+36wqx+18q^2x+36vy-36wp-36pq+36z)\partial_v+\\
        &(6vwx^3-18w^2x^2-12wqx^2+12wpx)\partial_w,\\
      %  \end{aligned}\end{align*}
 % \begin{align*}\begin{aligned}
        \tilde{Y}_7=& (12p^2-18qy)\partial_x+(8p^3-18pqy+18yz)\partial_y+(18pz-9q^2y)\partial_p+(18qz-6pq^2)\partial_q+\\&(18z^2-3q^3y)\partial_z+(18v^2y-18vwp-18vpq+9wq^2+3q^3+18vz)\partial_w+\\
        &(18vwy-18w^2p-12wpq+18wz)\partial_w.
\end{aligned}\end{align*}
  %As the lifts of $(X_1,\dots,X_7)$ with $k=2$ and $(Y_1,\dots,Y_7)$, t
  The 14 symmetries $(\tilde{X}_1,\dots,\tilde{X}_7,\tilde{Y}_1,\dots,\tilde{Y}_7)$ form a Lie algebra isomorphic to the split real form of the exceptional Lie algebra $\mathfrak{g}_2$.
\end{itemize}
\begin{itemize}
  \item The third group of 7 symmetries is given by 
 $$\begin{aligned}
        \hat{Z}_1=&\tfrac{1}{w}\Big(\partial_x+(p-wx)\partial_y+q\partial_p+v\partial_q+(qw+\tfrac12q^2)\partial_z\Big),\\
        \hat{Z}_2=&\tfrac{1}{w}\Big(x\partial_x+(px-\tfrac12wx^2)\partial_y+qx\partial_p+vx\partial_q+(wqx-wp+\tfrac12q^2x)\partial_z\Big)-\partial_w,\\
        \hat{Z}_3=&\tfrac{1}{w}\Big(q\partial_x+(pq-pw)\partial_y+q^2\partial_p+vq\partial_q+\tfrac12(q^2w+q^3)\partial_z\Big)-v\partial_w,\\
       \hat{Z}_4=&\tfrac{1}{w}\Big((2p-qx)\partial_x+(2p^2-pqx-3wy+wpx)\partial_y+q(2p-qx)\partial_p+v(2p-qx)\partial_q+\\&(2pq^2-q^3x-6zw+4pqw-q^2xw)\partial_z\Big)+(vx-3w-q)\partial_w,\\
    %   \end{aligned}$$
  %  $$\begin{aligned}
        \hat{Z}_5=&\tfrac{1}{w}\Big((4px-qx^2-6y)\partial_x+(wpx^2-pqx^2-3wxy+4p^2x-6py)\partial_y+\\&(3wpx-q^2x^2+4pqx-9wy-6qy)\partial_p+(4vpx-vqx^2+3wqx-6vy-6wp)\partial_q+\\&\tfrac12(8wpqx-q^3x^2-wq^2x^2+4pq^2x-4wp^2-12wqy-6q^2y)\partial_z\Big)+\\&(3vx-9w-3q)\partial_v+(vx^2-3wx-2qx+2p)\partial_w,\\
        \end{aligned}$$
    $$\begin{aligned}
        \hat{Z}_6=&\tfrac{1}{w}\Big((4pq-3wqx-12z)\partial_x+(3wxz-3wpqx-2wp^2+4p^2q-12pz)\partial_y+\\&\tfrac12(8pq^2-3wq^2x-18wz-24qz)\partial_p+(4vpq-3wq^2-12vz)\partial_q+\\&\tfrac12q(4wpq-wq^2x+4pq^2-24wz-12qz)\partial_z\Big)+(3v^2x+9vw-3vq)\partial_v+\\&(3vwx-4vp+9w^2+6wq+2q^2)\partial_w,\\
        \hat{Z}_7=&\tfrac{1}{w}\Big((3wqx^2-8pqx-18wy-16p^2+24qy+24xz)\partial_x+\\&(3wpqx^2+4wp^2x-3wx^2z-8p^2qx-24wpy-16p^3+24pqy+24pxz)\partial_y+\\&\tfrac12(3wq^2x^2-16pq^2x-36wp^2+36wxz-32p^2q+48q^2y+48qxz)\partial_p+\\&2(3wq^2x-4vpqx-8vp^2+12vqy+12vxz-9wpq+9wz)\partial_q+\\&\tfrac12(wq^3x^2-8wpq^2x-8pq^3x-32wp^2q+24wq^2y+48wqxz-16p^2q^2+\\&24q^3y+24q^2xz-12wpz)\partial_z\Big)+(6vqx-3v^2x^2-18vwx+18wq-3q^2)\partial_v+\\&(8vpx-3vwx^2-18w^2x-12wqx-4q^2x-24vy+30wp+16pq-24z)\partial_w.
  \end{aligned}$$
These symmetries are \emph{not} lifts of vector fields from $M$. In particular they are not lifts of conformal symmetries of the conformal class $[g_{{\mathcal D}_2}]$ of the distribution. In other words the algebra $\soa(3,4)$ of symmetries of the Lie conatct structure $([\lambda],[\Upsilon])$ on $\widetilde{M}=\mathbb{P}([{\mathcal D}_2,{\mathcal D}_2])\setminus\mathbb{P}({\mathcal D}_2)$ is \emph{not} a lift of the symmetry algebra $\soa(3,4)$ of the flat conformal structure $[g_{{\mathcal D}_2}]$ associated with the distribution ${\mathcal D}_2$.    
\end{itemize}
\subsubsection{Geometry on the boundary $\mathbb{P}(\mathcal{D}_2)$ of $\mathbb{P}([{\mathcal D}_2,{\mathcal D}_2])$}
Next we observe what happens if we pass to the $6$-dimensional boundary $\mathbb{P}(\mathcal{D}_2)$, which in our parametrization is given by  $w=0$. This is done by considering an inclusion \begin{align*} \iota:\mathbb{P}(\mathcal{D}_2)\hookrightarrow \mathbb{P}([\mathcal{D}_2,\mathcal{D}_2]),\quad \iota(x,y,p,q,z,v)=(x,y,p,q,z,v,0),\end{align*}  of the boundary $\mathbb{P}(\mathcal{D}_2)$ into $\mathbb{P}([\mathcal{D}_2,\mathcal{D}_2])$ and by pullbacking the structural objects $\lambda$ and $\Upsilon$ to the boundary. Taking $\lambda$ as in (\ref{cfk}) with $k=2$ gives:
$$\lambda_0=\iota^*\lambda=\der z-q\der p+v\der y+(\tfrac12q^2-vp)\der x.$$
This defines a 5-distribution ${\mathcal H}_0$ on $\mathbb{P}(\mathcal{D}_2)$ via
${\mathcal H}_0=\mathrm{ker}(\lambda_0)$.

Let us recall the following definition: Given a contact distribution ${\mathcal D}=\mathrm{ker}(\lambda)$ defined in terms of a 1-form $\lambda$ on a manifold $M$, a nonzero vector field $X$ on $M$ is called its \emph{Cauchy characteristic} if $X\hook\lambda=0$ and $X\hook\der\lambda=0\,\mathrm{mod}\, \lambda$. A Cauchy characteristic is a particular infinitesimal symmetry of $\mathcal D$, since the definition implies ${\mathcal L}_X\lambda\dz\lambda=0$. It follows that, in general, distributions have no Cauchy characteristics. However, it turns out that the distribution ${\mathcal H}_0$ on $\mathbb{P}(\mathcal{D}_2)$, has a Cauchy characteristic 
$$X=\partial_x +p\,\partial_y+q\, \partial_p+ \frac{q^2}{2}\partial_z+v\partial_q.$$
Interestingly this characteristic preserves $\Upsilon_0$ also, and we have ${\mathcal L}_{fX}\Upsilon_0=0$. To explicitly see this we adapt coordinates in such a way that five of them are invariant with respect to $X$ and the sixth one is choosen so that it ramifies $X$. Explicitly we pass from coordinates $(x,y,p,q,z,v)$ to coordinates $(x_0,x_1,x_2,x_3,x_4,x_5)$, where
\begin{align*}
&x=x_5\,, y=\tfrac{1}{6}x_1{x_5}^3+\tfrac{6^{1/3}}{2}x_2{x_5}^2+\tfrac{6^{2/3}}{2}x_3x_5+x_4,\, p=\tfrac{1}{2} x_1{x_5}^2+{\scriptstyle 6}^{1/3}x_2 x_5 +\tfrac{6^{2/3}}{2}x_3, \\& q=x_1 x_5+{\scriptstyle 6}^{1/3}x_2,\,z= \tfrac{6^{2/3}}{2} {x_2}^2x_5+\tfrac{6^{1/3}}{2} x_1 x_2{x_5}^2+\tfrac16{x_1}^2{x_5}^3+ x_0, \,v=x_1.
\end{align*}
In these new coordinates
\begin{align*}
X=\partial_{x_5},\quad\quad\quad\lambda_0=\der x_0-3x_2\der x_3+ x_1 \der x_4 ,
\end{align*}
and the pullback of the conformal symmetric rank 4 tensor is represented by
\begin{align*}
\Upsilon_0= -3 \der x_2^2 \der x_3^2 + 4 \der x_1 \der x_3^3 +4 \der x_2^3 \der x_4-6 \der x_1 \der x_2\der x_3\der x_4 + \der x_1^2 \der x_4^2 .
\end{align*}
This suggests to consider the five-dimensional quotient $N=\mathbb{P}(\mathcal{D}_2)/X$ of  $\mathbb{P}(\mathcal{D}_2)$ by the foliation given by $X$.
\subsubsection{Associated flat contact $G_2$ geometry in dimension 5}\label{associated_G_2cont}
The above formulae show that $\lambda_0$ and $\Upsilon_0$ descend to $N$. Moreover, we have \begin{align*}
\der\lambda_0\wedge\der\lambda_0\wedge\lambda_0=2 \der x_1\wedge\der x_2\wedge\der x_3\wedge\der x_4\wedge\der x_0,
\end{align*} so $\lambda_0$ defines a contact distribution $\mathcal{H}_0=\mathrm{ker} \lambda_0$ on $N$.  We equip this contact distribution with the line $[\Upsilon_0]$ of symmetric rank 4 tensors on $\mathcal{H}_0$ spanned by $\Upsilon_0$. Then one finds that the pointwise common stabilizer of $[\Upsilon_0]$ and $[(\der {\lambda_0})_{|\mathcal{H}_0}]$ is  isomorphic to $\mathrm{GL}(2,\mathbb{R})$ in the irreducible $4$-dimensional representation. That means that $([\lambda_0],[\Upsilon_0])$ describes a $G_2$-contact structure on $N$ as introduced in section \ref{G_2_contact}. 

The algebra of infinitesimal symmetries of the structure $([\lambda_0],[\Upsilon_0])$ is then defined as the set  of vector fields $X\in\mathfrak{X}(N)$ such that
\begin{align*}
(\mathcal{L}_X\lambda_0)\wedge\lambda_0=0,\quad \mathrm{and}\quad\mathcal{L}_X\Upsilon_0=f \Upsilon_0 +\lambda_0\odot\tau,
\end{align*}
 where $\tau$ is a rank $3$ tensor and $f$ is a function on $M$.
 We  calculated that the algebra of infinitesimal symmetries of $([\lambda_0],[\Upsilon_0])$ is the exceptional Lie algebra $\mathfrak{g}_2$, as described in the following proposition.

 \begin{proposition}\label{pr1}
All symmetries $X$ of the structure $([\lambda_0],[\Upsilon_0])$ defined by the representatives:
\be
\lambda_0=\der x_0-3x_2\der x_3+x_1\der x_4\label{la1}\ee
and
\be\Upsilon_0=-3\der x_2^2\der x_3^2+4\der x_1\der x_3^3+4\der x_2^3\der x_4-6\der x_1\der x_2\der x_3\der x_4+\der x_1^2\der x_4^2\label{up1}\ee
are $\bbR$-linear combinations of the following 14 vector fields:
$$
\begin{aligned}
X_1&=(x_0^2+3x_3^3x_1-3x_3x_1x_4x_2-x_4x_2^3-3x_3^2x_2^2)\partial_0+(x_1^2x_4+x_0x_1+x_2^3)\partial_1+\\&(2x_3x_2^2+x_2x_1x_4+x_0x_2-x_1x_3^2)\partial_2+(x_0x_3+x_2x_3^2+x_4x_2^2)\partial_3+\\&(x_0x_4+3x_3x_2x_4-x_3^3)\partial_4,\\
X_2&=-(x_0x_4-2x_3^3)\partial_0+(x_1x_4+x_0)\partial_1-x_3^2\partial_2-x_3x_4\partial_3-x_4^2\partial_4,\\
X_3&=-(\tfrac 12x_3x_1x_4+\tfrac12 x_4x_2^2+x_2x_3^2)\partial_0+\tfrac12 x_2^2\partial_1+(\tfrac23 x_3x_2+\tfrac16 x_1x_4+\tfrac16 x_0)\partial_2+\\&(\tfrac16 x_3^2+\tfrac13 x_2x_4)\partial_3+\tfrac12 x_3x_4\partial_4,\\
X_4&=-(x_2x_4+x_3^2)\partial_0+x_2\partial_1+\tfrac23 x_3\partial_2+\tfrac13 x_4\partial_3,\\
X_5&=-x_4\partial_0+\partial_1,\\
X_6&=(x_0 x_2-2 x_1 x_3^2)\partial_0+x_1 x_2 \partial_1+(\tfrac13 x_2^2+\tfrac23 x_3 x_1) \partial_2+(\tfrac13 x_3 x_2-\tfrac13 x_0) \partial_3+x_3^2 \partial_4,\\
X_7&=x_0\partial_0+x_1\partial_1+\tfrac23 x_2\partial_2+\tfrac13 x_3\partial_3,\\
X_8&=-(\tfrac32 x_3x_1x_2+\tfrac12 x_2^3)\partial_0+\tfrac12 x_1^2\partial_1+\tfrac12 x_1x_2\partial_2+\tfrac12 x_2^2\partial_3+(\tfrac32 x_3x_2+(\tfrac12 x_0)\partial_4\\
X_9&=-(x_3x_1+x_2^2)\partial_0+\tfrac13 x_1\partial_2+\tfrac23 x_2\partial_3+x_3\partial_4,\\
X_{10}&=\partial_4,\\
X_{11}&=x_0\partial_0+\tfrac13 x_2\partial_2+\tfrac23 x_3\partial_3+x_4\partial_4,\\
X_{12}&=-3x_2\partial_0+\partial_3\\
X_{13}&=\partial_2\\
X_{14}&=\partial_0.
\end{aligned}
$$
Here the symbols $\partial_\mu$ denote the partial derivatives with respect to the variables $x_\mu$: $\partial_\mu=\tfrac{\partial}{\partial x_\mu}$, $\mu=0,1,2,3,4$. 
The Lie algebra generated by the 14 vector fields $X_A$, $A=1,2,\dots, 14$ is isomorphic to the split real form of the exceptional simple Lie algebra ${\mathfrak g}_2$, and thus the $G_2$ contact structure $([\lambda_0],[\Upsilon_0])$ is flat.
\end{proposition}

\section{$\mathrm{G}_2$-reduced Lie contact structures}\label{secReduction}
Here we  show  that the Lie contact structures  on $\mathbb{P}([{\mathcal D},{\mathcal D}])\setminus\mathbb{P}({\mathcal D})$ associated with $(2,3,5)$-distributions $\mathcal{D}$
 have holonomy  reduced to $\mathrm{G}_2\subset \mathrm{O}(3,4)$. We  further study, more generally, Lie contact structures in dimension $7$ whose holonomy is reduced to $\mathrm{G}_2$. In particular, we  prove Proposition \ref{prop-hol} and Theorem \ref{thm1.2}.

\subsection{Normality of the induced Cartan connection} We start this chapter with a technical result: we will prove that the induced Lie contact Cartan connection $\widetilde{\omega}$ satisfies the normality condition $\widetilde{\partial}^*\widetilde{\omega}=0$.
Note that we did not need this information to show that the twistor bundle of a $(2,3,5)$-distribution carries an induced Lie contact structure.  However, the fact that  $\widetilde{\omega}$  is the canonical normal Cartan connection 
 will be of importance for further applications, in particular  Proposition \ref{prop-hol}.
 
 Given Theorem \ref{g2_gs}, proving normality of $\widetilde{\omega}$  is a straightforward task, although computationally involved.
The following alternative proof uses methods from parabolic geometry,
in particular Kostant's theorem  \cite{kostant} and  Corollary 3.2 in \cite{cap-correspondence}, which we will use to derive information about the full curvature  of  regular, normal parabolic geometries associated with $(2,3,5)$-distributions from information about their harmonic curvature space. 

The Kostant codifferential can be written in terms of basis as follows: Let $X_1,\dots,X_n\in\g$ project to a basis for $\g/\p$ and let $Z_1,\dots,Z_n\in\mathfrak{p}_+\cong (\g/\mathfrak{p})^*$ the dual basis, then for any $\phi\in\Lambda^2(\g/\p)^*\otimes\g$ and $X\in\g$,
\begin{align*}
\partial^*\phi (X)=2\sum_{i}[\phi(X_i,X),Z_i]+\sum_{i}\phi(X_i,[Z_i,X]),
\end{align*}
see \cite[Lemma 3.1.11]{cap-slovak-book}.

\begin{lemma}\label{normal}
Suppose $(\mathcal{G}\to M,\omega)$ is a regular and normal parabolic geometry of type $(\mathrm{G}_2,P)$, then the induced parabolic geometry $(\widetilde{\mathcal{G}}\to \widetilde{M},\widetilde{\omega})$  of type $(\mathrm{O}(3,4),\widetilde{P})$ is normal.
\end{lemma}

\begin{proof} Let $K:\mathcal{G}\to\Lambda^2(\g/\p)^*\otimes\tilde{\g}$ be the curvature function of $\omega$ and let  $\widetilde{\mathcal{G}}\to\Lambda^2(\tilde{\g}/\tilde{\p})^*\otimes\tilde{\g}$ be the curvature function of $\widetilde{\omega}$. By  $\widetilde{P}$-equivariancy of $\widetilde{\partial}^*\widetilde{K}$ it suffices to prove that $\widetilde{\partial}^*\widetilde{K}(u)=0$ for any $u\in\mathcal{G}$ (rather than  $u\in\widetilde{\mathcal{G}}$) in order to show that the induced geometry is normal. Recall that, for $u\in\mathcal{G}$,  
\begin{align*}
\widetilde{K}(u)=(\Lambda^2\varphi\otimes i')(K(u)),
\end{align*}
where $i':\mathfrak{g}\to \tilde{\mathfrak{g}}$ is the Lie algebra inclusion and  $\varphi:(\mathfrak{g}/\mathfrak{p})^*\to (\tilde{\mathfrak{g}}/\tilde{\mathfrak{p}})^*$ is the dual map to the projection $\tilde{\mathfrak{g}}/\tilde{\mathfrak{p}}\cong \mathfrak{g}/\mathfrak{q}\to\mathfrak{g}/\mathfrak{p}$.

Next let us  recall some facts from the general theory of parabolic geometries, see  \cite{kostant, cap-correspondence} for details.
One can, as a $G_0$-representation, identify the harmonic curvature space $\mathrm{ker}(\partial^*)/\mathrm{im}(\partial^*)$ with the kernel of the so-called Kostant Laplacian $\mathrm{ker}(\square)\subset\Lambda^2\p_{+}\otimes\g$. A lowest weight vector of $\mathrm{ker}(\square)$  can be algorithmically  determined using Kostant's theorem. Consider the grading  \eqref{g2grad} of $\g$, then  in our case the lowest weight vector  is an element of the form 
$$\phi_1=Z_1\wedge Z_4\otimes A\in \g_1\wedge\g_3\otimes\g_0.$$
Now, since regular, normal parabolic geometries of type $(\mathrm{G}_2,P)$ are torsion-free,    Corollary 3.2 from \cite{cap-correspondence} implies that the curvature function $K$ takes values  in the $P$-module generated  by 
successively raising
 this  lowest weight vector. Note that this implies, for instance, that  $K(u)(X,Y)=0$ whenever both $X$ and $Y$ are contained in $\g^{-2}$. (Of course, as mentioned earlier, we can also read off this information from Theorem \ref{g2_gs}.)

Now to prove the lemma, pick an arbitrary map $\phi\in\Lambda^2\p_+\otimes\g$ contained in the $P$-module generated by raising  the lowest weight vector in $\mathrm{ker}(\square)$; in particular $\partial^*\phi=0$. Let 
\begin{align*}
\widetilde{\phi}=(\Lambda^2\varphi\otimes i')(\phi) 
\end{align*}
be the corresponding element in $\Lambda^2\tilde{\p}_{+}\otimes\tilde{\g}$.
Choose
% the grading  \eqref{g2grad} of $\g$ and 
elements
$X_1,X_2\in\mathfrak{g}_{-1}$, $X_3\in\mathfrak{g}_{-2}$, $X_4,X_5\in\mathfrak{g}_{-3}$ defining a basis for $\g/\p$,  supplement them by  $X_6,X_7\in\g_1$ to obtain a basis for $\tilde{\g}/\tilde{\p}\cong\g/\q$. Use the Killing form on $\tilde{\g}$, which restricts to a multiple of the Killing form on $\g$, to identify $\p_+\cong(\g/\p)^*$ and $\tilde{\p}_+\cong(\tilde{\g}/\tilde{\p})^*\cong(\g/\q)^*$, and let $Z_1,\dots,Z_5\in\p_+$ and $\widetilde{Z}_1,\dots,\widetilde{Z}_7$ be the respective dual bases.
By construction $\widetilde{\phi}$  vanishes upon insertion of elements of $\p$, hence $\widetilde{\phi}(\cdot,X_i)=0$ for $i=6,7$. Thus, \begin{align*}
\widetilde{\partial}^*\widetilde{\phi}(X)=2\sum_{i=1,..,5}[\widetilde{Z}_i,\widetilde{\phi}(X,X_i)]-\sum_{i=1,..,5}\widetilde{\phi}([\widetilde{Z}_i,X],X_i)
\end{align*}
for any $X\in\g$. Using that $\partial^* \phi=0$ this can also be written as 
%=\underbrace{2\sum_{i=1,..,5}[Z_i,K(X,X_i)]-\sum_{i=1,..,5}K([Z_i,X],X_i)}_{=0}+\\
\begin{align}\label{diffnor}
\widetilde{\partial}^*\widetilde{\phi}(X)=2\sum_{i=1,..,5}[\widetilde{Z}_i-Z_i,\widetilde{\phi}(X,X_i)]-\sum_{i=1,..,5}\widetilde{\phi}([\widetilde{Z}_i-Z_i,X],X_i).
\end{align}

Let us first show that the second term in the above expression vanishes. Note that $\tilde{\g}$ splits into the direct sum of $\g\subset\tilde{\g}$ and its orthogonal complement with respect to the Killing form $\g^{\perp}\subset\tilde{\g}$, which can be identified as a $\g$-representation with  the $7$-dimensional fundamental representation $\mathbb{V}$ of $\g$. By construction, the differences $\widetilde{Z}_i-Z_i$ are contained in the orthogonal complement to $\g$, i.e. in $\mathbb{V}=\g^{\perp}$. Now $\mathbb{V}$ is $\g$-invariant, hence $[\widetilde{Z}_i-Z_i,X]\subset\mathbb{V}$ for any $X\in\g$. 
More precisely,  $\tilde{Z}_i-Z_i\in\mathbb{V}_1$ for $i=1,2$, $\tilde{Z}_2-Z_2\in\mathbb{V}_2$, and $\tilde{Z}_i-Z_i=0$ for $i=4,5$, where we use the grading from \eqref{g2grad}. Moreover, 
%\begin{align*}
$\mathbb{V}=\bigoplus_{i=-2,\dots,2}\mathbb{V}_{i}\subset\g^{-2}+\tilde{\p}.$
%\end{align*}
Since $\phi(X,Y)=0$ if both $X$ and $Y$ are contained in $\g^{-2}$, this implies that
%\mathbb{V}_{-2}\oplus\cdots\oplus\mathbb{V}_3$
$$\sum_{i=1,..,5}\widetilde{\phi}\, ([\widetilde{Z}_i-Z_i,X],X_i)=0.$$

Now for the first term in \eqref{diffnor}, consider the $\g_0$-invariant decomposition of $\Lambda^2\p_{+}\otimes\g$ according to homogeneity with respect to the grading $\eqref{g2grad}$ on $\g$ (in the sense that an  element $\phi\in\g_i\wedge\g_j\otimes\g_k$ has homogeneity $i+j+k$). Since $[\g_i,\mathbb{V}_j]\subset\mathbb{V}_{i+j}$ and $\mathbb{V}^3=\bigoplus_{i\geq 3}\mathbb{V}_i= \{0 \}$, one sees that
\begin{align*}
\sum_{i=1,..,5}[\widetilde{Z}_i-Z_i,V]=0\ \, \text{and} \ \sum_{i=3,..,5}[\widetilde{Z}_i-Z_i,W]=0.
\end{align*} 
for any $V\in\g^2=\g_2\oplus\g_3$ and $W\in\g^1=\g_1\oplus\g_2\oplus\g_3$. Keeping in mind also that  $\widetilde{Z}_4-Z_4=\widetilde{Z}_5-Z_5=0$ and that $\phi(X,Y)=0$ if both $X$ and $Y$ are contained in $\g^{-2}$, one concludes  that it remains to inspect $\phi$'s contained in the $P$-module generated by the lowest weight vector $\phi_1$ intersected with
\begin{align*}
\g_1\wedge\g_3\otimes\g_0\oplus \g_2\wedge\g_3\otimes\g_0 \oplus \g_1\wedge\g_3\otimes\g_1
\end{align*}
(i.e., of homogeneity $4$ or $5$). Indeed, by Schur's Lemma and since $G_0$ includes into $\tilde{P}$, it suffices to compute $\widetilde{\partial}^*\widetilde{\phi}$ for one representative $\phi$ in each irreducible $G_0$-submodule of that space. One easily sees that there are only two such $G_0$-submodules: 
The lowest weight vector $$\phi_1=Z_1\wedge Z_4\otimes A$$ generates the first one,  and raising it we obtain a  generator of the second one of the form 
$$\phi_2=Z_3\wedge Z_4\otimes A + Z_1\wedge Z_4\otimes Z_1;$$ here $Z_1\in\g_1$, $Z_3\in\g_2$, $Z_4\in\g_3$, $A\in\g_0$  and, since  $\partial^*\phi_1=\partial^*\phi_2=0$, $[Z_1, A]=[Z_3, A]=[Z_4,A]=0$. Using that  this implies that $[\tilde{Z}_1, A]=[\tilde{Z}_3, A]=[\tilde{Z}_4,A]=0$ and the facts $\tilde{Z}_4-Z_4=0$ and $[Z_1,\tilde{Z}_1]=0,$ which can be verified directly, we immediately conclude that the corresponding elements 
$\widetilde{\phi}_1=\tilde{Z}_1\wedge \tilde{Z}_4\otimes A$ and 
$\widetilde{\phi}_2=\tilde{Z}_3\wedge\tilde{Z}_4\otimes A + \tilde{Z}_1\wedge \tilde{Z}_4\otimes Z_1$
are contained in the kernel of $\widetilde{\partial}^*$.
%the map defined by the first summand of  \eqref{diffnor}. 
This completes the proof.

\end{proof}

\subsection{Holonomy in $\mathrm{G}_2$ and a parallel tractor $3$-form}
Let $(\mathcal{G}\to M, \omega)$ be a Cartan geometry of type $(G,P)$ and let $\hat{\omega}$
be the canonical extension of $\omega$ to a  principal connection on the extended $G$-principal bundle $\hat{\mathcal{G}}:=\mathcal{G}\times_{P}G.$ Assume that $M$ is connected. The holonomy group of the Cartan geometry at a point $u\in\hat{\mathcal{G}}$ is then  defined to be the the \emph{holonomy group} $$\mathrm{Hol}_u(\omega):=\mathrm{Hol}_u(\hat{\omega})\subset G$$ of the principal connection $\hat{\omega}$ at that point. Since different choices of base points $u$ lead to conjugate subgroups within $G$, we will disregard the base point and speak of  the holonomy  $\mathrm{Hol}(\omega)$  of the Cartan connection $\omega$ (keeping in mind that it is well-defined only up to conjugacy in $G$).  If $(\mathcal{G}\to M, \omega)$ is a normal, regular parabolic geometry  encoding an underlying structure (e.g. a $(2,3,5)$-distribution or a Lie contact structure)  then the holonomy of the underlying structure is defined to be the holonomy of the associated normal Cartan connection.

Holonomy reductions of Cartan connections are related to parallel sections of so-called tractor bundles. Given a   $G$-representation $\mathbb{W}$,  the principal connection $\hat{\omega}\in\Omega^1(\hat{\mathcal{G}},\g)$  induces a linear connection $\nabla$ on the associated bundle 
$$\mathcal{W}:=\mathcal{G}\times_{P}\mathbb{W}=\hat{\mathcal{G}}\times_{G}\mathbb{W}.$$
Vector bundles arising that way are called \emph{tractor bundles} and the induced linear connections are called \emph{tractor connections}. If the Cartan connection $\omega$ is normal, the induced tractor connection is said to be normal. By definition of $\mathcal{W}$ as an associated bundle, sections $s\in\Gamma(\mathcal{W})$ correspond to smooth equivariant maps $f_s:\hat{\mathcal{G}}\to\mathbb{W}$. A section $s$ is parallel for the tractor connection if and only if the corresponding function is constant along all horizontal curves $c:I\to\hat{\mathcal{G}}$,\, $\hat{\omega}(c'(t))=0$. The holonomy group $\mathrm{Hol}(\omega)$ is then contained in the pointwise stabilizer of the parallel section $s$.

Now consider a Lie contact structure of signature $(1,2)$ on a manifold $\widetilde{M}$ 
%and let $(\widetilde{\mathcal{G}}\to \widetilde{M},\tilde{\omega})$ be the 
with associated regular, normal parabolic geometry of type $(\mathrm{O}(3,4),\widetilde{P})$. Let $\mathbb{V}$ be the standard representation for $\mathrm{O}(3,4)$ and $\mathcal{T}$ the associated tractor bundle with its normal tractor connection. %Consider the (unique up to constants) $SO(3,4)$-invariant  bilinear form and volume form on $\mathbb{V}$.  
The constant map $f_{H}$
 from the Cartan bundle
 onto the (unique up to constants) $\mathrm{O}(3,4)$-invariant  bilinear form  defines a parallel section $\mathbf{H}\in\Gamma(S^2\mathcal{T}^*)$ called the \emph{tractor metric}.

Next recall the following (well-known) characterization of the Lie group $\mathrm{G}_2$. Consider a $7$-dimensional vector space $\mathbb{V}$ with bilinear form $H$ of signature $(3,4)$. Let $\Phi\in\Lambda^3\mathbb{V}^*$ be a $3$-form, then
% \begin{align*}
 $(X, Y)\mapsto (X\hook\Phi)\wedge (Y \hook\Phi)\wedge\Phi$
 %\end{align*}
 defines a symmetric $\Lambda^7\mathbb{V}^*$-valued bilinear form on $\mathbb{V}$. If this bilinear form is non-degenerate, then it determines a volume form $\mathrm{vol}_{\Phi}$ and thus a $\mathbb{R}$-valued symmetric bilinear form $H_{\Phi}$. 
 Now suppose that $H_{\Phi}$ is a multiple of $H$, i.e.
  \begin{align}\label{comp}H_{\Phi}(X,Y)\mathrm{vol}_{\Phi}:=  (X\hook\Phi)\wedge (Y\hook\Phi)\wedge\Phi=\lambda \, H(X,Y)\mathrm{vol}_{\Phi},\end{align}
 for a constant $\lambda$. Then the stabilizer of $\Phi$ is a copy of $\mathrm{G}_2\subset \mathrm{SO}(3,4)=\mathrm{SO}(H)$.
  %as defined in section \ref{}.
   We will call a $3$-form satisfying the above condition \emph{compatible}, and we will use the same terminology on the level of tractors.

 As an immediate consequence of the construction and Lemma \ref{normal}, we obtain the following:

\begin{corollary}
 The Lie contact structure on $\widetilde{M}$ induced by a $(2,3,5)$-distribution  $\mathcal{D}$ admits a compatible tractor $3$-form $\mathbf{\Phi}\in\Gamma(\Lambda^3\mathcal{T}^*)$ which is parallel for the normal tractor connection, and the holonomy of the Lie contact structure reduces to $\mathrm{G}_2$.
 \end{corollary}

\begin{proof}
 Let $(\mathcal{G}\to M,\omega)$ be the regular, normal parabolic geometry of type $(\mathrm{G}_2,P)$ associated with the $(2,3,5)$-distribution $\mathcal{D}$. Let  $(\widetilde{\mathcal{G}}\to \widetilde{M},\widetilde{\omega})$ be the induced parabolic geometry of type $(\mathrm{O}(3,4),\widetilde{P})$ on the twistor bundle. Then, by construction, the  principal connection $\hat{\widetilde{\omega}}$ on the extended bundle $\widetilde{\mathcal{G}}\times_{\tilde{P}}\mathrm{O}(3,4)$ reduces to  the $\mathrm{G}_2$-principal bundle connection $\hat{\omega}$ on $\mathcal{G}\times_{P}\mathrm{G}_2$.
 
 Now let $\Phi\in\Lambda^3\mathbb{V}^*$ be a defining $3$-form for $\mathrm{G}_2\subset \mathrm{O}(3,4)$. Then the constant $\mathrm{G}_2$-equivariant  map $f_{\Phi}:\hat{\mathcal{G}}\to\Lambda^3\mathbb{V}^*$ onto $\Phi$ defines a section $\mathbf{\Phi}\in\Gamma(\Lambda^3\mathcal{T}^*)$ of the Lie contact tractor bundle, which is compatible with $\mathbf{H}$. Since $\hat{\widetilde{\omega}}$ is the extension of the $G_2$-principal connection $\hat{\omega}$, $\mathbf{\Phi}$ is  parallel for the  tractor connection induced by  $\hat{\widetilde{\omega}}$, and by Lemma \ref{normal} this is the normal tractor connection on $\Lambda^3\mathcal{T}^*$.
Moreover, 
$$\mathrm{Hol}(\widetilde{\omega})=\mathrm{Hol}(\hat{\widetilde{\omega}})\subset \mathrm{G}_2\subset \mathrm{O}(3,4)$$ and, 
again by normality of $\widetilde{\omega}$, 
this is the holonomy of the underlying Lie contact structure. 

\end{proof}
 In particular, we have proven Proposition \ref{prop-hol}.

\subsection{A curved orbit decomposition}
Next we consider the more general situation of a  Lie contact structure of signature $(1,2)$
%and let $\mathcal{T}\to\widetilde{M}$ be the standard tractor bundle with tractor metric $H$, tractor volume form $\mathrm{vol}$ and normal tractor connection $\widetilde{\nabla}$. Suppose further
 together with a  tractor $3$-form $\mathbf{\Phi}\in\Gamma(\Lambda^3\mathcal{T}^*)$ that is compatible in the sense of \eqref{comp} and parallel for the normal tractor connection. %such that the associated  metric is the given tractor metric, i.e. $H_{\Phi}=H$ and $\tilde{\nabla}\Phi=0$.
 Then the pointwise stabilizer of $\mathbf{\Phi}$ is ${\mathrm{G}_2}$ and 
 the holonomy  of the Lie contact structure is reduced, 
$\mathrm{Hol}(\widetilde{\omega})\subset \mathrm{G}_2.$
In order to formulate the geometric implications of this set-up, we will 
apply
%employ  some more tools from parabolic geometry, in particular 
the curved orbit decomposition theorem  discussed below.

 Let $(\mathcal{G}\to M, \omega)$ be a Cartan geometry of type $(G,P)$ and let $s\in\Gamma(\mathcal{W})$ be a parallel section of some tractor bundle $\mathcal{W}$ with corresponding $G$-equivariant function $f_{s}:\hat{\mathcal{G}}\to \mathbb{W}$. Assuming that $M$ is connected, the image $f_s(\hat{\mathcal{G}})$ is a $G$-orbit $\mathcal{O}\subset\mathbb{W}$.
 In \cite{CGH-holonomy} the following pointwise invariant of $s$ is introduced: the image  $f_s(\mathcal{G}_x)\subset \mathcal{O}$ of a fibre is a $P$-orbit called the $P$-type of $x$ with respect to $s$. The manifold $M$ then decomposes according to the $P$-type of points
 into a disjoint union of \emph{curved orbits} $M_i$,  $$M=\bigsqcup_{i\in P\setminus \mathcal{O}}M_i,$$ where  $P\setminus \mathcal{O}$ denotes the set of $P$-orbits of the $G$-orbit $\mathcal{O}$.
 Fix an element in $\mathcal{O}$ and let $H\subset G$ be its stabilizer.
% ; this provides an identification  $\mathcal{O}\cong G/H$. 
 Then the  set of $P$-orbits of $\mathcal{O}\cong G/H$ is in bijective correspondence with the set of $H$-orbits of $G/P$ via $PgH\mapsto Hg^{-1}P$. In particular, the set of curved orbits can be parametrized by $H$-orbits of $G/P$. Now suppose that $M_i$ is a non-empty curved orbit and let $\alpha_i$ be a representative of the corresponding $H$-orbit $H\cdot \alpha_i\subset G/P$. Then it is  shown in \cite{CGH-holonomy} that:
  \begin{itemize}
 \item for any $x\in M_i$ there are neighbourhoods $U\subset M$ of $x$  and $V\subset G/P$ of $\alpha_i$  and a diffeomorphism $\psi:U\to V$ such that $\psi(U\cap M_i)=V\cap (H\cdot \alpha_i)$. 
 %In particular $M_i$ is an initial submanifold of the same dimension as the corresponding orbit.
 \item $M_i$ carries an induced Cartan geometry $(\mathcal{G}_i\to M_i,\omega_i)$  of the same type as the corresponding $H$-orbit in $G/P$. The Cartan bundle can be realized as a subbundle $\mathcal{G}_i\subset  \mathcal{G}\vert_{M_i}$  and the Cartan connection $\omega_i$ is the pullback of $\omega$ with respect the corresponding inclusion.
 %$\omega$  reduces  to a Cartan connection of type $(H,H_i)$ on a subbundle $\mathcal{G}_i\subset  \mathcal{G}\vert_{M_i}$ over $M_i$, where $H_i\subset H$ is the stabilizer of $\alpha_i\in G/P$ in $H$.
 \end{itemize}

In the following we apply this result in the case of interest for us, i.e., when the Cartan geometry is of type $(\mathrm{O}(3,4),\widetilde{P}),$ the section $s=\mathbf{\Phi}\in\Gamma(\Lambda^3\mathcal{T}^*)$ is a parallel compatible tractor $3$-form and the stabilizer $H=\mathrm{G}_2$. As before,  $\widetilde{P}\subset \mathrm{O}(3,4)$ denotes the Lie contact parabolic, $P\subset \mathrm{G}_2$ the $(2,3,5)$ parabolic and $\bar{P}\subset \mathrm{G}_2$ the $\mathrm{G}_2$-contact parabolic as introduced in Section \ref{parabolics}.
 
  \begin{theorem}\label{thmHolred}
Suppose $\widetilde{M}$ is a $7$-manifold endowed with a Lie contact structure of signature $(1,2)$ and let $(\widetilde{\mathcal{G}}\to\widetilde{M},\widetilde{\omega})$ be the corresponding regular, normal parabolic geometry. Let $\mathbf{\Phi}\in\Gamma(\Lambda^3\mathcal{T}^*)$ be a parallel compatible tractor $3$-form that defines a holonomy reduction to $\mathrm{G}_2$. 

Then the corresponding curved orbit decomposition is of the form
\[\widetilde{M}=\widetilde{M}^o\cup\widetilde{M}',\]
where $\widetilde{M}^o$ is open and $\widetilde{M}'$ (if non-empty) is a $5$-dimensional submanifold of $\widetilde{M}$. 
\begin{enumerate}
\item If $\widetilde{M}'$ is non-empty, then it carries an induced $\mathrm{G}_2$-contact structure.

\item $\widetilde{M}^o$ carries an induced Cartan geometry $(\mathcal{G}\to \widetilde{M}^o,\omega^o)$ of type $(G,Q)$.   
Suppose further that the curvature of this Cartan geometry satisfies $K^o(u)(X,Y)=0$ for all $X\in\p$ and $Y\in\g$. Then the rank $2$ bundle $\mathcal{V}\subset T\widetilde{M}^o$ corresponding to $\p/\q$ is integrable and around each point $x\in\widetilde{M}^o$ we can form a local   $5$-dimensional leaf space which inherits a $(2,3,5)$ distribution. 
%The corresponding twistor Lie contact structure is locally isomorphic to the  Lie contact structure we started out with.
\end{enumerate}
\end{theorem}

\begin{proof}
 The first statement is an immediatete consequence of  Proposition \ref{orbits}, which describes the $\mathrm{G}_2$-orbit decomposition of $\mathrm{O}(3,4)/\widetilde{P}$,  and the curved orbit decomposition theorem. Combining these results  shows that the manifold $\widetilde{M}$ decomposes into an open submanifold $\widetilde{M}^o$ and a  complement $\widetilde{M}'$, which is  either empty or a $5$-dimensional submanifold. 
 %The theorem further shows that the curved orbits  inherit Cartan geometries of the same type as the corresponding orbits in the homogeneous model:  
 $\widetilde{M}'$ carries  an induced  Cartan geometry $(\mathcal{G}'\to \widetilde{M}',\omega')$ of type $(G,\bar{P})$ and $\widetilde{M}^o$ carries an induced  Cartan geometry  $(\mathcal{G}^o\to \widetilde{M}^o,\omega^o)$ of type $(G,Q).$ These can be realized as subbundles in $\widetilde{\mathcal{G}}\vert_{\widetilde{M}'}$ and $\widetilde{\mathcal{G}}\vert_{\widetilde{M}^o},$ respectively, and the Cartan connections, $\omega'$ and $\omega^o$, and their curvatures, $K'$ and $K^o$, are the pullbacks of the Cartan connection $\widetilde{\omega}$ and curvature $\widetilde{K}$ with respect to the inclusions.

Using this, we next show that the induced Cartan connection on $\widetilde{M}'$ is regular. First,  the curvature $\widetilde{K}$ of the regular, normal Lie contact Cartan connection takes values in $\Lambda^2(\tilde{\g}/\tilde{\p})^*\otimes\tilde{\g}^{-1},$ which follows from the structure of the harmonic curvature $\widetilde{K}_{H}$ and an application of the Bianchi identity. 
This implies that the curvature $K'$ of the reduced connection   takes values in $\Lambda^2(\g/\bar{\p})^*\otimes(\tilde{\g}^{-1}\cap\g)$. Now $\tilde{\g}^{-1}\cap\g$ coincides with the filtration component $\bar{\g}^{-1}$ for the $\mathrm{G}_2'$ contact grading \eqref{g2bargrad}, and this implies that $K'$ is of homogeneity $\geq 1$, i.e.  the Cartan connection $\omega'$ is regular. In particular, $\widetilde{M}'$ carries an induced $\mathrm{G}_2'$ contact structure.

Next we investigate the Cartan geometry of type $(G,Q)$ on $\widetilde{M}^o$. Via the Cartan connection, the $Q$-submodule $\p/\q\subset \g/\q$ determines a distinguished rank $2$ subbundle $\mathcal{V}$ in $T\widetilde{M}^o$. Now suppose that the curvature function of  satisfies $K^o(u)(X,\cdot)=0$ for all $X\in\p$. It is proven in \cite{cap-correspondence}, see also Theorem 1.5.14 in \cite{cap-slovak-book}, that this implies that the subbundle $\mathcal{V}$ is integrable, and  locally around each point one can form a corresponding leaf space $M$, which carries an induced Cartan geometry  of type $(\mathrm{G}_2,P)$.

To see that the Cartan geometry  of type $(\mathrm{G}_2,P)$ determines a $(2,3,5)$-distribution on the leaf space $M$, it remains to see that the Cartan connection is regular. Arguing as before shows that the $Q$-equivariant curvature function $K^o$ takes values in  $\Lambda^2(\g/\q)^*\otimes(\tilde{\g}^{-1}\cap\g)$. Looking at the gradings \eqref{sograd} and \eqref{g2grad} shows that $\tilde{\g}^{-1}\cap\g=\g_{-3}\oplus\g_{-1}\oplus\p$. Note that this space is a $Q$-module, but not a $P$-module. The condition $K^o(u)(X,\cdot)=0$ for all $X\in\p/\q$ in particular implies that, locally, $\mathcal{G}\to M$ is a $P$-principal bundle  and the curvature function $K^o$ is $P$-equivariant. Now suppose that for some $u\in\mathcal{G}$ and $X,Y\in \g$,  $K^o(u)(X,Y)$ has a non-trivial component in $\g_{-3}$. Then we can find some $g\in\mathrm{exp}(\g_1)\subset P$ such that  $K^o(u\cdot g^{-1})(\mathrm{Ad}(g)\cdot X,\mathrm{Ad}(g)\cdot Y)=\mathrm{Ad}(g)\cdot K^o(u)(X,Y)$ has a non-trivial  component in $\g_{-2}$. But this is a contradiction to the assumptions on the values of $K^o$. Hence under the additional curvature condition, the curvature function takes indeed values in $\Lambda^2(\g/\p)^*\otimes \g^{-1}$, which implies that the  $(\mathrm{G}_2,P)$ geometry on $M$ is regular.
\end{proof}

\begin{remark}
One can show that the resulting Cartan connection $\omega\in\Omega^1(\mathcal{G}\to M,\g_2)$  is indeed the normal Cartan connection associated with the induced distribution on the local leaf space. However, this requires more  information on the curvature of the regular, normal Lie contact Cartan connection, and will be discussed elsewhere.
\end{remark}
  
 \begin{remark}
The decomposition into curved orbits can also be described using the so-called 
%next ingredient that we will use is that of 
normal BGG solution determined by the parallel tractor $3$-form $\mathbf{\Phi}$.

Recall (see Section \ref{parabolics}) that the parabolic subgroup $\widetilde{P}$ preserves a filtration
%$\mathbb{E}\subset\mathbb{E}^{\perp}\subset\mathbb{V}$
$\widetilde{\mathbb{V}}^{-1}\supset \widetilde{\mathbb{V}}^{0}\supset\widetilde{\mathbb{V}}^{1}$ 
of the standard representation, where $\widetilde{\mathbb{V}}^{1}=\mathbb{E}$, $\widetilde{\mathbb{V}}^{0}=\mathbb{E}^{\perp}$ and $\widetilde{\mathbb{V}}^{-1}=\mathbb{V}$.
 Correspondingly,  the standard tractor bundle is filtered
 $$\mathcal{T}\supset \mathcal{T}^{0}\supset\mathcal{T}^{1}$$
%$$\mathcal{E}\subset \mathcal{E}^{\perp}\subset\mathcal{T},$$
where  $\mathcal{T}^{1}=E$,  $\mathcal{T}^{0}/\mathcal{T}^{1}\cong F$ and $\mathcal{T}/\mathcal{T}^{1}\cong E^*$.  
 There is an induced filtration of $\Lambda^3\mathcal{T}^*,$
  and a natural projection onto the quotient by the largest proper subbundle in this filtration, $$\Pi:\Lambda^3\mathcal{T}^*\to \Lambda^3\mathcal{T}^*/(\Lambda^3\mathcal{T}^*)^0\cong\Lambda^2 E^*\otimes F.$$
  % be the natural  projection. 
 The image of a tractor $\Phi\in\Gamma(\Lambda^3\mathcal{T}^*)$ under this projection defines an element $$\phi\in\Gamma(\Lambda^2 E^*\otimes F),$$ i.e. a weighted section of $F$. 
 By the general theory of parabolic geometries,  if $\Phi$ is a \emph{parallel} tractor $3$-form,  then the underlying section  $\phi\in\Gamma(\Lambda^2 E^*\otimes F)$ is contained in the kernel of a  first order linear differential operator, called \emph{first BGG operator for $\Lambda^3\mathcal{T}^*$}. Solutions of the corresponding overdetermined system of PDEs that are obtained in that way are called normal $BGG$ solutions. See \cite{BGG-2001, CGH-holonomy} for more details.

Now suppose that $\mathbf{\Phi}\in\Gamma(\Lambda^3\mathcal{T}^*)$ is a parallel compatible tractor $3$-form. Recall (see Section \ref{parabolics} and Proposition \ref{orbits}) that inserting a totally null $2$-plane $\mathbb{E}$ into  a  defining  $3$-form for $\mathrm{G}_2$ gives either zero or a null line $\ell\in\mathbb{E}^{\perp}$ transversal to $\mathbb{E}$.  Hence, for a parallel compatible tractor $3$-form $\Phi$, at any point $x\in M$ either $\phi_x=0$  or $\phi_x$ defines a null line in $F$ with respect to $b$. The decomposition of $\widetilde{M}$ into $\widetilde{P}$-types of $\Phi$ corresponds to the decomposition into the zero locus $M'$ of $\phi$ and the open subset  $M^o$ where $\phi$ is nonvanishing.
On $M^o$ , via the isomorphism $\mathcal{H}=E^*\otimes F$, the filtration $\phi\subset\phi^{\perp}\subset F$ determines
a distinguished filtration of a rank $2$ subbundle contained in a rank $4$ subbundle contained in the contact subbundle 
$$\mathcal{V}_{\phi}\subset \mathcal{D}_{\phi}\subset\mathcal{H}.$$ 
Looking at the explicit matrices \eqref{sograd} and \eqref{g2grad}, it can be seen that via the isomorphisms $\g/\q\cong \tilde{\g}/\tilde{\p}$ and  $\tilde{\g}^{-1}/\tilde{\p}\cong \mathbb{E}^*\otimes \mathbb{E}^{\perp}/\mathbb{E}$,  the subspace $\p/\q$ corresponds to $\mathbb{E}^*\otimes \ell$, where $\ell$ now denotes the projection of ${\mathbb{E}}\hook\mathbf{\Phi}$ to $\mathbb{E}^{\perp}/\mathbb{E}$, and $\g^{-1}/\q$ corresponds to $\mathbb{E}^*\otimes \ell^{\perp}$. Hence, for Lie contact structures coming from $(2,3,5)$ distributions via the twistor construction, $\mathcal{V}_{\phi}$ is the vertical bundle for the projection $\widetilde{M}\to M$, and  $\mathcal{D}_{\phi}=\widetilde{\mathcal{D}}$ projects to the downstairs $(2,3,5)$-distribution.
%\end{remark}

%\begin{remark}
The twistorial construction of Lie contact structures from $(2,3,5)$-distributions provides many non-flat examples of holonomy reductions to $\mathrm{G}_2$. However, by construction, in these cases the  corresponding parallel tractor $3$-form $\mathbf{\Phi}$ has only one $\widetilde{P}$-type and the underlying $BGG$ solution $\phi$ is  nowhere vanishing.
%and so we do not get curved orbit decompositions  that way.
It would be interesting to see if one can find non-flat examples admitting $\phi$'s with non-empty zero sets  that carry induced $\mathrm{G}_2$-contact structures.
\end{remark}

%\bibliographystyle{plain}
%\bibliography{Bibliography.bib}

%\begin{thebibliography}{10}

\def\polhk#1{\setbox0=\hbox{#1}{\ooalign{\hidewidth
  \lower1.5ex\hbox{`}\hidewidth\crcr\unhbox0}}}

\end{document}